\documentclass[11pt]{article}
\usepackage{graphicx}
\usepackage{graphics}
\usepackage{amsfonts}
\usepackage{amscd}
\usepackage{amsthm}
\usepackage{indentfirst}
\usepackage[all]{xy}
\usepackage{amssymb, amsmath, amsthm, amsgen, amstext, amsbsy, amsopn}
\usepackage{epic,eepic}
\usepackage{hyperref}
\usepackage{enumitem} 
\usepackage{color}
\usepackage{ dsfont }
\usepackage[margin=1.3in]{geometry}

\theoremstyle{plain}
\newtheorem*{thm*}{Theorem}
\newtheorem{thm}{Theorem}
\newtheorem{lemma}{Lemma}[section]
\newtheorem{prop}[lemma]{Proposition}
\newtheorem{claim}[lemma]{Claim}
\newtheorem{cor}[lemma]{Corollary}

\theoremstyle{definition}
\newtheorem{defin}[lemma]{Definition}
\newtheorem{rem}[lemma]{Remark}
\newtheorem{example}[lemma]{Example}
\newtheorem{conj}[lemma]{Conjecture}
\newtheorem{exam}[lemma]{Example}

\newcommand{\R}{{\mathbb{R}}}

\newcommand{\Z}{{\mathbb{Z}}}
\newcommand{\N}{{\mathbb{N}}}
\newcommand{\C}{{\mathbb{C}}}

\newcommand{\cA}{{\mathcal{A}}}

\newcommand{\cC}{{\mathcal{C}}}

\newcommand{\cL}{{\mathcal{L}}}

\newcommand{\cN}{{\mathcal{N}}}
\newcommand{\cP}{{\mathcal{P}}}

\newcommand{\cU}{{\mathcal{U}}}

\def\id{{1\hskip-2.5pt{\rm l}}}

\newcommand{\indCZ}{{\operatorname{CZ}}}
\newcommand{\indRS}{{\operatorname{RS}}}

\newcommand{\Sp}{{\operatorname{Sp}}}

\newcommand{\spec}{{\operatorname{Spec}}}

\newcommand{\sign}{{\operatorname{sign}\,}}
\newcommand{\im}{{\operatorname{Im}\,}}

\begin{document}
\title{A max inequality for spectral invariants of disjointly supported Hamiltonians.}
\author{Shira Tanny}
\maketitle
\begin{abstract}
	We study the relation between spectral invariants of disjointly supported Hamiltonians and of their sum. On aspherical manifolds, such a relation was established by Humili\`ere, Le Roux and Seyfaddini. We show that a weaker statement holds in a wider setting, and derive applications to Polterovich's Poisson bracket invariant and to Entov and Polterovich's notion of superheavy sets.	
\end{abstract}

\tableofcontents
\section{Introduction and results.}

Hamiltonian spectral invariants on closed symplectic manifolds were introduced by Oh and Schwartz \cite{oh2005construction,schwarz2000action}. These invariants assign to each Hamiltonian $H:M\times S^1\rightarrow\R$ and a non-zero quantum homology class $\alpha\in QH_*(M)$ a real number, denoted by $c(H;\alpha)$. In this paper we consider only spectral invariants with respect to the fundamental class, and therefore abbreviate $c(\cdot):=c(\cdot;[M])$.
Spectral invariants have been widely studied and have many applications in symplectic  geometry. 
One relevant application concerns lower bounds for Polterovich's Poisson bracket invariant, which was introduced in \cite{polterovich2012quantum,polterovich2014symplectic}. Given a finite open cover $\{U_i\}_{i=1}^N$ of a closed symplectic manifold, the Poisson bracket invariant of $\{U_i\}$ is defined by
\begin{equation*}
pb(\{U_i\}):=\inf_{\{f_i\}}\ \max_{|x_j|,|y_k|\leq 1}\left\{\sum_j x_j f_j, \sum_k y_k f_k\right\},
\end{equation*} 
where the infimum is taken over all smooth partitions of unity that are subordinate\footnote{A partition of unity is a collection of non-negative functions that sum up to 1. We say that $\{f_i\}$ is subordinate to $\{U_i\}$ if $supp(f_i)\subset U_i$ for each $i$.} to the cover $\{U_i\}$. 
Polterovich explained the relation of this invariant to quantum mechanics and conjectured a lower bound for it, in terms of the displacement energies of the sets. Moreover, he showed how upper bounds for the spectral invariants of sums of disjointly supported Hamiltonians can be used to establish lower bounds for $pb$.  This inspired several works studying upper bounds for the spectral invariant of a sum of disjointly supported Hamiltonians: In  \cite{polterovich2014symplectic}, Polterovich produced upper bounds for Hamiltonians supported in certain domains on symplectically aspherical manifolds, namely when both the symplectic form $\omega$ and the first Chern class $c_1$ vanish on $\pi_2(M)$. Later, in \cite{seyfaddini2014spectral}, Seyfaddini constructed so called {\it spectral killers} and bounded the spectral invariant of a sum of Hamiltonians supported in disjoint small balls on monotone manifolds, i.e., when $\omega$ is proportional to $c_1$. In \cite{ishikawa2015spectral}, Ishikawa considered Hamiltonians supported in symplectic embeddings of strongly convex sets in $\R^{2n}$, into monotone manifolds. 
Finally, in \cite{humiliere2016towards}, Humili\`ere, Le Roux and Seyfaddini proved that, on symplectically aspherical manifolds, the spectral invariant of a sum of Hamiltonians supported in certain disjoint domains, is equal to the maximum over the spectral invariants of the Hamiltonians:
\begin{thm*}[Humili\`ere-Le Roux-Seyfaddini, \cite{humiliere2016towards}]
	Let $H_1,\dots,H_N$ be  Hamiltonians supported in disjoint incompressible Liouville domains in a symplectically aspherical manifold. Then,
	\begin{equation*}
	c(H_1+\cdots+H_N)=\max\{c(H_1),\dots,c(H_N)\}.
	\end{equation*}
\end{thm*} 
This result is referred to as the ``max formula" for spectral invariants.
{An alternative proof for the max formula, as well as an inequality for spectral invariants with respect to a general homology class, were given in \cite{ganor2020floer}.}
Humili\`ere, Le Roux and Seyfaddini also showed that the max formula does not hold on the sphere, by constructing Hamiltonians $H_1$ and $H_2$, supported in disjoint disks on $S^2$, for which $c(H_1+H_2)<\max\{c(H_1),c(H_2)\}$. A natural question is whether an inequality holds in general. In what follows, we consider disjointly supported Hamiltonians $H_1,\dots, H_N$ on a closed connected symplectic manifold $(M,\omega)$, and show that under certain conditions one has 
\begin{equation}\label{eq:max_ineq}
c(H_1+\cdots+H_N)\leq \max\{c(H_1),\dots,c(H_N)\}.
\end{equation}
The main ingredient of the proof is the construction of spectral killers, in the spirit of Seyfaddini \cite{seyfaddini2014spectral}. We change Seyfaddini's construction in order to prove a max inequality, as well as extend it to a more general setting.

\subsection{The max inequality for disjointly supported Hamiltonians.}
Let $(M,\omega)$ be a closed symplectic manifold. Throughout the paper, we consider Hamiltonians supported in domains satisfying certain conditions. These domains include, for example, symplectic embeddings into $M$ of star-shaped domains in $\R^{2n}$ with smooth boundaries and such that the radial vector field is transverse to the boundary (following \cite{gutt2018symplectic}, we call such domains ``nice star-shaped domains"). In order to describe the class of relevant domains in full generality, let us recall a few standard notions.
A domain $U\subset M$ has a {\it contact type boundary} if there exists a vector field $Y$, called {\it the Liouville vector field}, that is defined on a neighborhood of the boundary $\partial U$, satisfies $\cL_Y\omega=\omega$, is transverse to the boundary, and points outwards.
In this case, $\lambda:=\iota_Y\omega$ is a primitive of $\omega$ and its restriction to $\partial U$ is called the {\it contact form} associated to $Y$. The {\it Reeb vector field} $R$ on $\partial U$ is defined by the equations
$$\omega(Y,R)=1,\quad \omega(R,\cdot)|_{T\partial U}=0.$$
The flow $\varphi_R^t:\partial U\rightarrow\partial U$ of $R$ is called the {\it  Reeb flow}, and we denote the set of its contractible periodic orbits (of any period) by $\cP(\partial U)$.  The action of a periodic Reeb orbit $\gamma\in\cP(\partial U)$ is given by $\int_\gamma \lambda$. The {\it action spectrum} of $\partial U$ is  
$$
\spec(\partial U) := \left\{\int_\gamma\lambda:\gamma\in \cP(\partial U)\right\}.
$$
Finally, the boundary $\partial U \subset M$ is called {\it incompressible} if the map $\pi_1(\partial U)\rightarrow\pi_1(M)$ induced by the inclusion is injective. In particular, when $\partial U$ is simply connected, it  is incompressible. 

We prove that the max inequality (\ref{eq:max_ineq}) holds for Hamiltonians supported in disjoint domains with incompressible contact type boundaries, under additional conditions on the Hamiltonians and the domains, which depend on whether the symplectic manifold is positively monotone, negatively monotone or rational. For the sake of convenience we assume from now on that $\dim M=2n$ is greater than 2, unless stated otherwise. The max formula proved by Humili\`ere, Le Roux and Seyfaddini holds for all symplectic surfaces other than the sphere. We discuss the max inequality on the sphere in Section~\ref{subsec:pos_mon_intro} which concerns positively monotone manifolds. Let us start by describing the results on rational symplectic manifolds.

\subsubsection{Rational manifolds.}
Let $(M,\omega)$ be a closed rational symplectic manifold, namely $\omega(\pi_2(M))=\kappa\Z$ for some $\kappa\in \R$.
{It is simpler to establish a max inequality if the disjoint supports are ``far enough" from one another. In order to make this condition precise consider the following definition.

\begin{defin}\label{def:extendable}
	Let $U\subset M$ be a domain with a contact type boundary. We say that $U$ is {\it $\sigma$-extendable}, for $\sigma>0$, if the  flow $\psi^\tau$ of the Liouville vector field $Y$ exists for all time $0<\tau<\log(1+\sigma)$. The {\it $\sigma$-extension} of such a domain $U$ is  defined to be
	\begin{equation*}
	(1+\sigma) U:=U\cup \left(\bigcup_{\tau\in[0,\log(1+\sigma)]}\psi^\tau\partial U\right).
	\end{equation*}
\end{defin}
See Figure~\ref{fig:extendable} for an illustration of a $\sigma$-extendable domain. We remark that for every domain $U$ with a contact type boundary, there exists $\varepsilon>0$ such that $U$ is $\varepsilon$-extendable, see Section~\ref{subsec:preliminaries_Liouville_coord_and_Reeb_flow}.
\begin{figure}
	\centering
	\includegraphics[scale=0.9]{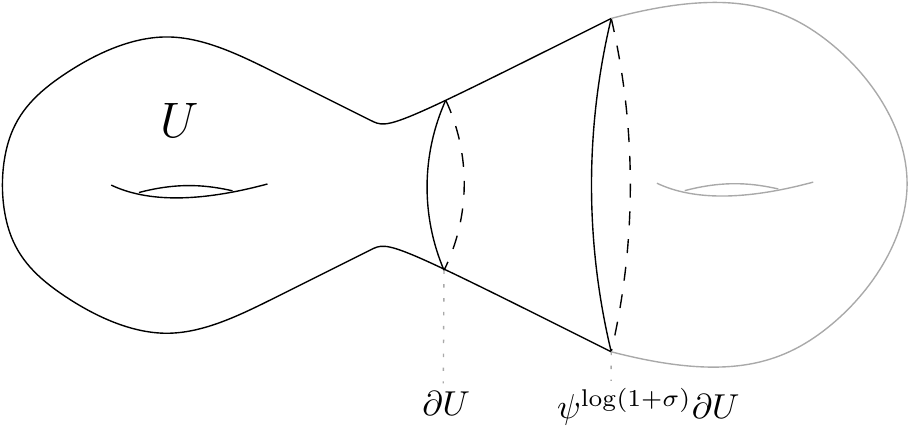}
	\caption{\small{An illustration of a $\sigma$-extendable domain with a contact type boundary. 
	}}
	\label{fig:extendable}
\end{figure}
\begin{exam}
 {Suppose that the ball $B$ of radius $r$ around the origin in $\R^{2n}$ (endowed with the standard symplectic form) embeds into $M$. Then the restriction of this embedding to the ball of radius $r/\sqrt{2}$ in $B$ is a $1$-extendable domain.} 
\end{exam}

The following theorem asserts that the max inequality holds for Hamiltonians which are supported in extendable domains with disjoint extensions (see Figure~\ref{fig:disjoint_extendable}), and whose spectral invariants are small compared to the ``size" of the extensions.

\begin{thm}\label{thm:max_ineq_rational_extendable}
	Let $U_i$ be $\sigma_i$-extendable domains with incompressible contact type boundaries, such that the extensions $\{(1+\sigma_i) U_i\}$ are pairwise disjoint.  Then, for Hamiltonians $H_i$ supported in $U_i$, such that $c(H_i)< \min\{\kappa, \sigma_i\cdot \min \spec(\partial U_i)\}$ it holds that 
	\begin{equation*}
	c(H_1+\cdots+H_N)\leq \max\{c(H_1),\dots,c(H_N)\}.
	\end{equation*}
\end{thm}

\begin{figure}
	\centering
	\includegraphics[scale=0.9]{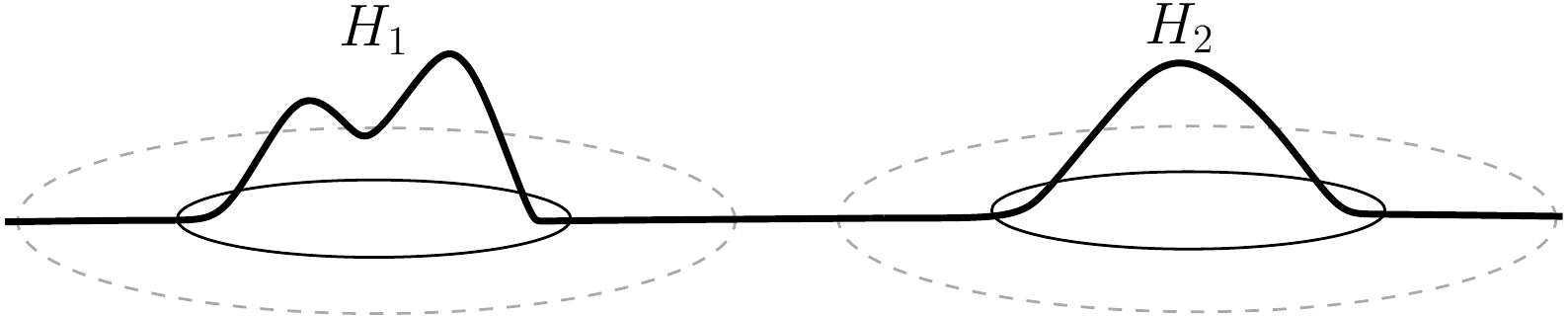}
	\caption{\small{An illustration of Hamiltonians supported in extendable domains, such that the extensions are disjoint.
	}}
	\label{fig:disjoint_extendable}
\end{figure}

When the domains containing the supports are not necessarily ``far" from each other, we assume that} the action spectrum of the contact boundaries, $\spec(\partial U)$, is contained in a lattice $T \Z$, such that $T$ divides $\kappa$. Examples for such domains are symplectic embeddings of balls of radius $r$ in $\R^{2n}$ such that $\pi r^2$ divides $\kappa$.
\begin{thm}\label{thm:max_ineq_rational}
	Let $U_i$ be disjoint domains with incompressible contact type boundaries such that $\spec(\partial U_i)\subset T_i\Z$ and  $T_i|\kappa$ for all $i$. Then, for Hamiltonians $H_i$ supported in $U_i$ such that $c(H_i)<T_i$, it holds that
	\begin{equation}\label{eq:sigma_extension}
	c(H_1+\cdots+H_N)\leq \max\{c(H_1),\dots,c(H_N)\}.
	\end{equation}
\end{thm}

\subsubsection{Negatively monotone manifolds.}
On negatively monotone manifolds,  namely when $\omega|_{\pi_2(M)} = \kappa\cdot c_1|_{\pi_2(M)}$ for $\kappa\leq 0$, we show that the max inequality (\ref{eq:max_ineq}) holds for Hamiltonians supported in disjoint domains with incompressible contact type boundaries, if we assume in addition that the {\it contact Conley-Zehnder index} of every Reeb orbit is non-negative. The contact Conley-Zehnder index assigns an integer, which we denote by $\indCZ_R(\gamma, u)$, to every periodic Reeb orbit $\gamma\in\cP(\partial U)$ and a capping disk $u\subset \partial U$. This index is well defined when the Reeb flow is {\it non-degenerate}\footnote{We remark that one can perturb the Liouville vector field to make the Reeb flow non-denegerate, see  Section~\ref{sec:preliminaries}.}, and is reviewed in Section~\ref{sec:preliminaries} together with other preliminaries from symplectic and contact geometry.
It is well known that when $U$ is a symplectic  embedding of a strictly  convex domain in $\R^{2n}$ into $M$, the CZ index of every Reeb orbit is non-negative\footnote{In \cite{hofer1998dynamics}, Hofer, Wysocki and Zehnder proved that for every strictly convex domain in $\R^{2n}$ with a smooth boundary, the contact CZ index of every Reeb orbit is at least $n+1$, and in particular is positive (the definitions and proofs are written for $n=2$, see the remark on p.222 for the general case).}. 
\begin{thm}\label{thm:max_ineq_neg_monotone}
	Let $(M,\omega)$ be a negatively monotone symplectic manifold and suppose $U_i\subset M$ are disjoint domains with incompressible contact type boundaries, such that the contact Conley-Zehnder index of every Reeb orbit is non-negative.
	Then, for any collection of Hamiltonians $H_i:M\times S^1\rightarrow\R$ supported in $U_i$ respectively,
	\begin{equation*}
	c(H_1+\cdots+H_N)\leq \max\{c(H_1),\dots,c(H_N)\}.
	\end{equation*}
\end{thm}

\subsubsection{Positively monotone manifolds.}\label{subsec:pos_mon_intro}
When the symplectic manifold is positively monotone,  namely $\omega|_{\pi_2(M)} = \kappa\cdot c_1|_{\pi_2(M)}$ for $\kappa\geq 0$, we need 
to impose additional assumptions on the domains $U_i$ and the Hamiltonians $H_i$, in order prove the max inequality (\ref{eq:max_ineq}).
The first requirement is that $U_i$ are {\it dynamically convex}
\footnote{{In fact, it is sufficient to assume that the contact CZ index of every Reeb orbit is at least $n$.}}%
, namely, that the contact Conley-Zehnder index of every Reeb orbit (with respect to a capping disk that is contained in the boundary) 
is at least $n+1$, where $n$ is half the dimension of $M$. It is known that every strictly convex domain in $\R^{2n}$ with a smooth boundary is dynamically convex, see, e.g., \cite{hofer1998dynamics}. Secondly, we require that the spectral invariants of the Hamiltonians are smaller than the monotonicity constant, namely, $c(H_i)<\kappa$.
Finally, we assume that the domains $U_i$ are ``not too big" compared to the monotonicity constant $\kappa$. The size of a domain is measured by maximizing the action-index ratio over all Reeb orbits on the boundary: 
\begin{defin}\label{def:ratio_invariant}
	Let $U\subset M$ be a domain with an incompressible contact type boundary, such that the Reeb flow is non-degenerate. 
	We define
	\begin{equation}
	C(U):=\sup\left\{\frac{2\int_{\lambda} \gamma}{\indCZ_R( \gamma, u_0)-n+1}:\gamma\in\cP(\partial U),\ u_0\subset\partial U\right\}\in \R\cup \{+\infty\}.
	\end{equation}
\end{defin}

\begin{rem}\label{rem:CZ_ratio_constant}
	The above definition can be extended to disjoint unions of domains $U=\sqcup_i U_i$. In this case, the invariant will be equal to the maximum over the invariants of the connected components, $C(U) = \max_i C(U_i)$. Definition~\ref{def:ratio_invariant} can be also extended to domains with degenerate Reeb flows. This is done in Section~\ref{sec:the_invariant_C}, together with estimates of the invariant $C$ on certain classes of domains: 
	\begin{itemize}
		\item Using results from \cite{gutt2018symplectic}, we show that for concave toric domains (and, in particular, ellipsoids) $C(U)$ coincides with the Gromov width\footnote{The Gromov width of $U$ is defined to be the supremum of $\pi r^2$ over all radii $r$ such that the ball of radius $r$ in $\R^{2n}$ (equipped with the standard symplectic form) can be symplectically embedded into $U$.}, $c_G(U)$. For convex toric domains, $c_G(U)\leq C(U)\leq c_G(B)$ for every ball $B$ whose image under the moment map contains the image of $U$, namely $\mu(B)\supset \mu(U)$. We review the definitions of convex and concave toric domains in Section~\ref{sec:the_invariant_C}. 
		\item Using a result by Ishikawa from \cite{ishikawa2015spectral}, we show that for strictly convex domains,  $C(U)$ can be bounded in terms of the curvature of the boundary $\partial U$. 
	\end{itemize}
\end{rem}

\begin{thm}\label{thm:max_ineq_pos_monotone}
Let $(M,\omega)$ be a positively monotone symplectic manifold of dimension greater than 2, and with monotonicity constant $\kappa>0$. Suppose that $U_i\subset M$ are disjoint domains with incompressible dynamically convex boundaries, such that $C(U_i)\leq\kappa$ for all $i$. For Hamiltonians $H_i:M\times S^1\rightarrow\R$ supported in $U_i$ respectively, such that $c(H_i)<\kappa$, we have
\begin{equation*}
	c(H_1+\cdots+H_N)\leq \max\{c(H_1),\dots,c(H_N)\}.
\end{equation*}
\end{thm}

The condition $c(H_i)<\kappa$ in the above theorem can be guaranteed if, for example, the supports are displaceable with small displacement energy, as follows from the energy-capacity inequality. This inequality, as well as  the definitions for displaceability and displacement energy, are stated in  Section~\ref{sec:preliminaries}. Theorem~\ref{thm:spec_bound_pos} below states that, in the setting of Theorem~\ref{thm:max_ineq_pos_monotone}, the condition $c(H_i)<\kappa$ holds if we assume in addition that $U_i$ are {\it portable Liouville domains}:
\begin{defin}[Following \cite{polterovich2014symplectic}] \label{def:portable}
	\begin{itemize}
		\item A domain $U$ with a contact type boundary is called a {\it Liouville domain} if the Liouville vector field $Y$ extends to $U$
		and satisfies $\cL_Y\omega=\omega$ there. 
		\item The  {\it core} of a Liouville domain $U$ is defined to be $Q:=\cap_{s\in(0,1]} \psi^{\log s}U$, where $\{\psi^\tau\}_{\tau\leq 0}$ is the flow of the Liouville vector field.
		\item A Liouville domain $U$ is called {\it portable} if $Q$ is displaceable in $U$.
		\item The {\it portability number} of $U$ is defined to be
		\begin{equation}
		p(U):=\lim_{s\rightarrow 0} e( \psi^{\log s}U;U)/s,
		\end{equation}
		where $e( \psi^{\log s}U;U)$ is the displacement energy of $ \psi^{\log s}U$ inside $U$.
	\end{itemize}
\end{defin}
\begin{example}
	Every nice star-shaped domain in $(\R^{2n},\omega_0)$ is a portable Liouville domain, and its portability number is equal to its displacement energy, $p(U)=e(U;\R^{2n})$.
\end{example}

The next corollary follows from Theorem~\ref{thm:max_ineq_pos_monotone} together with Theorem~\ref{thm:spec_bound_pos} below, which gives an upper bound for the spectral invariants of Hamiltonians supported in portable Liouville domains with dynamically convex incompressible boundaries.
\begin{cor}\label{cor:max_ineq_pos_mon}
	Let $(M,\omega)$ be a positively monotone manifold with monotonicity constant $\kappa$, of dimension greated than 2. Suppose that $U_i\subset M$ are disjoint portable Liouville domains with incompressible dynamically convex boundaries, such that $C(U_i)\leq\kappa$ for all $i$. Then, for any collection of Hamiltonians $H_i$ supported in $U_i$ respectively, the max inequality holds: 
	\begin{equation*}
	c(H_1+\cdots+H_N)\leq \max\{c(H_1),\dots,c(H_N)\}.
	\end{equation*}
\end{cor}

The last statement of this section establishes the max inequality for Hamiltonians supported in certain disks on the sphere.
\begin{claim}\label{clm:sphere}
	Let $(S^2,\omega)$ be the sphere with area normalized to 1. Let $H_i$ be Hamiltonians supported on disjoint disks $D_i\subset S^2$ such that  $area(D_i)\notin(1/3,1/2)$ for each $i$. Then, 
	\begin{equation*}
	c(H_1+\cdots+H_N)\leq \max\{c(H_1),\dots,c(H_N)\}.
	\end{equation*}
\end{claim}
{The method of our proof of the claim does not apply when the condition on the area of the disks is not satisfied. To deduce the max inequality we need to bound the actions of certain Reeb orbits on the boundary of the domain away from the spectral invariants of the considered Hamiltonians. The area of the disk determines the actions of the Reeb orbits on its boundary. 
We remark that in \cite{seyfaddini2014spectral}, Seyfaddini considered balls in monotone manifolds whose displacement energy is bounded by half the monotonicity constant. On $S^2$ (with total area normalized to 1) this amounts to disks of area less than $1/4$. In this setting Seyfaddini proved that the spectral invariant of a sum of Hamiltonians supported in such balls is bounded by the maximal capacity of these balls. Corollary~\ref{cor:max_ineq_pos_mon} and Claim~\ref{clm:sphere} together with Theorem~\ref{thm:spec_bound_pos} can be thought of as an extension of the results from \cite{seyfaddini2014spectral} for positively monotone manifolds, where Theorems~\ref{thm:max_ineq_neg_monotone} and \ref{thm:spec_bound_neg} extend the results to the setting of negatively monotone manifolds.}

\subsection{Applications for the Poisson bracket invariant.}
The main application of the max inequality (\ref{eq:max_ineq}) concerns the Poisson bracket invariant of covers, which was defined by Polterovich in \cite{polterovich2014symplectic}. As explained above, this invariant assigns a non-negative number, $pb(\cU)$, to a finite open cover $\cU=\{U_i\}$ of a closed symplectic manifold. The Poisson bracket invariant is known to be strictly positive when the cover consists of displaceable sets. Polterovich conjectured a lower bound for the Poisson bracket invariant:
\begin{conj}[Polterovich,  \cite{polterovich2014symplectic}]
	Let $(M,\omega)$ be a closed symplectic manifold. There exists a constant $c_M$, depending only on the symplectic manifold $(M,\omega)$, such that for every finite open cover $\cU=\{U_i\}$ of $M$,
	\begin{equation*}
	pb(\cU)\geq \frac{c_M}{e(\cU)},
	\end{equation*}  
	where $e(\cU):=\max_i e(U_i)$ is the maximal displacement energy of a set from $\cU$.
\end{conj}
This conjecture was proved for the case where $M$ is a surface in \cite{buhovsky2020poisson}, and for surfaces other than the sphere in \cite{payette2018geometry}. In higher dimensions the conjecture is still open and all known lower bounds for $pb$ decay with the {\it degree} of the cover \cite{polterovich2014symplectic, seyfaddini2014spectral, ishikawa2015spectral}. The degree of an open cover $\cU:=\{U_i\}_{i=1}^N$ is defined to be the maximal number of sets intersected by a single set:
\begin{equation*}
d(\cU):=\max_i \#\{j:\bar U_i\cap \bar U_j \neq \emptyset\}.
\end{equation*}
In \cite{polterovich2014symplectic}, Polterovich proved that on  symplectically aspherical manifolds
$$
pb(\cU)\geq c_M/(d(\cU)^2\cdot \max_i p(U_i))
$$
for every cover $\cU$ consisting of portable Liouville domains. Here $p(U_i)$ is the portability number of $U_i$, from Definition~\ref{def:portable}. Later, Seyfaddini in \cite{seyfaddini2014spectral} proved that on monotone manifolds, i.e. when $\omega|_{\pi_2(M)}=\kappa c_1|_{\pi_2(M)}$, one has $pb(\cU)\geq 1/(2d(\cU)^2\cdot \max_i c_G(U_i))\geq 1/(2d(\cU)^2\cdot e(\cU))$ for every cover $\cU$ by balls that are displaceable with energy smaller than $|\kappa|/2$. Finally, in \cite{ishikawa2015spectral}, Ishikawa gave a lower bound for covers consisting of embeddings of strictly convex sets into monotone manifolds, which decays quadratically in the degree and depends on the curvature of the boundaries. The max inequality yields lower bounds in terms of the displacement energies of the sets, still decaying with the degree. The following corollary follows from Theorems~\ref{thm:max_ineq_neg_monotone} and \ref{thm:max_ineq_pos_monotone} together with Corollary~\ref{cor:max_ineq_pos_mon}, by arguments that appear in \cite{polterovich2014symplectic,seyfaddini2014spectral,polterovich2014function}. More specifically, we refer the reader to the proof of \cite[Theorem 9]{seyfaddini2014spectral}.
\begin{cor}\label{cor:pb_mon}
Let $(M,\omega)$ be a monotone symplectic manifold with monotonicity constant $\kappa$, and let $\cU:=\{U_i\}_{i=1}^N$ be a finite open cover of $M$ by domains with incompressible dynamically convex boundaries. Assume in addition that one of the following holds: 
\begin{itemize}
	\item $\kappa\leq 0$
	\item $\kappa> 0$, $C(U_i)\leq \kappa$ for all $i$ and, for each $i$,  either $e(U_i)< \kappa$ or $U_i$ is a portable Liouville domain.
\end{itemize}
Then,
\begin{equation}\label{eq:pb_monotone}
pb(\cU)\geq \frac{1}{2\cdot d(\cU)^2 \cdot e(\cU)}.
\end{equation}
\end{cor}
\begin{rem}
	When the cover $\cU$ consists of portable Liouville domains with incompressible dynamically convex boundaries, the maximal displacement energy, $e(\cU)$, in the lower bound (\ref{eq:pb_monotone}) can be replaced by $\max_ip(U_i)$ if $\kappa\leq 0$, and by $\max_iC( U_i)$ otherwise. This follows from the proof of \cite[Theorem 9]{seyfaddini2014spectral} together with Theorems \ref{thm:spec_bound_pos} and \ref{thm:spec_bound_neg} below, which give uniform bounds for the spectral invariants of Hamiltonians supported in such sets. In this case, one obtains a positive lower bound for the Poisson bracket invariant when the cover does not necessarily consist of displaceable sets.
\end{rem}

{We can use Theorem~\ref{thm:max_ineq_rational_extendable} to deduce a lower bound for $pb$ for certain covers on rational manifolds. Following the notations of Definition~\ref{def:extendable} above, assume that $\cU$ is a cover by 1-extendable balls and notice that $2\cU:=\{2 U_i\}$ is also a cover of $M$ by symplectically embedded balls. When the symplectic manifold $(M,\omega)$ is rational and the capacity of each ball in the cover is not greater than the rationality constant {$\kappa$}, the Poisson bracket invariant of $\cU$ can be bounded from below using the degree of the cover $2\cU$.} 
\begin{cor}\label{cor:pb_rational}
	Let $(M,\omega)$ be a rational manifold with rationality constant $\kappa$, and let $\cU$ be a cover by {1-extendable} balls, such that $c_G(U_i)\leq \kappa$. Then,
	\begin{equation*}
	pb(\cU)\geq\frac{1}{2\cdot d(2\cU)^2\cdot e(\cU)}. 
	\end{equation*} 
\end{cor}
Corollary~\ref{cor:pb_rational} can be deduced from Theorem~\ref{thm:max_ineq_rational_extendable} in the same way that \cite[Theorem 9]{seyfaddini2014spectral} is deduced from Theorem 2 there, together with the following observation. {Every 1-extendable ball $U$ is displaceable with energy $e(U)=c_G(U)$. Every Hamiltonian $H$ that is compactly supported in $U$ is also supported in a slightly smaller ball $U'$, of capacity strictly less than $c_G(U)$. By the energy capacity inequality, $c(H)\leq e(U')=c_G(U')<c_G(U)=\min\{\kappa, 1\cdot \min\spec(\partial U)\}$, which guarantees the requirements of Theorem~\ref{thm:max_ineq_rational_extendable}.}
	
\subsection{Uniform bounds on spectral invariants and super heavy sets.}
Theorem~\ref{thm:max_ineq_pos_monotone} above states that, on positively monotone manifolds, the max inequality holds for Hamiltonians supported in certain domains, assuming that their spectral invariants are smaller than the monotonicity constant $\kappa$. The following theorem guarantees that this bound on the spectral invariants holds under additional conditions on the domains containing the supports.
\begin{thm}\label{thm:spec_bound_pos}
	Let $(M,\omega)$ be a positively monotone symplectic manifold and suppose $U\subset M$ is a disjoint union of portable Liouville domains with dynamically convex incompressible boundaries. If $C(U)\leq\kappa$, then for every Hamiltonian $H:M\times S^1\rightarrow\R$ supported in $U$, $c(H)<C(U)$.
\end{thm}

The next result can be seen as a generalization to negatively monotone manifolds of \cite[Proposition 5.4]{polterovich2014symplectic}, which states that the spectral invariant of every Hamiltonian supported in a portable Liouville domain on a symplectically aspherical manifold is bounded by the portability number.
\begin{thm}\label{thm:spec_bound_neg}
	Let $(M,\omega)$ be a negatively monotone symplectic manifold and suppose that $U\subset M$ is a disjoint union of portable Liouville domains with dynamically convex incompressible boundaries. Then for every Hamiltonian $H:M\times S^1\rightarrow\R$ supported in $U$, its spectral invariant is bounded by the portability number of $U$, namely $c(H)\leq p(U)$.
\end{thm}
Here the portability number of a disjoint union of portable Liouville domains is defined to be the maximal portability number of a connected component, namely $p(U):=\max_i p(U_i)$. Theorems~\ref{thm:spec_bound_pos} and \ref{thm:spec_bound_neg} can be seen as versions of Ishikawa's result, \cite[Proprosition 4.4]{ishikawa2015spectral}, with different upper bounds and constraints. Ishikawa proved that, on negatively monotone manifolds, the spectral invariant of every Hamiltonian supported in a disjoint union of embeddings of strictly convex sets into $M$, is bounded by a constant depending on the curvature of the boundary. On positively monotone manifolds, he gave a different upper bound which also depends on the curvature, under the assumption that the minimal curvature is not too small, compared to the monotonicity constant.
As explained in \cite{ishikawa2015spectral}, an immediate corollary of Theorems~\ref{thm:spec_bound_pos} and \ref{thm:spec_bound_neg} concerns the notion of a {\it superheavy} set, which was introduced by Entov and Polterovich in \cite{entov2009rigid}:
A closed subset $X\subset M$ is called superheavy if 
\begin{equation}\label{eq:super-heavy}
\lim_{k\rightarrow\infty}\frac{c(kH)}{k} \leq \sup_{X\times S^1} H, \quad \forall H\in\cC^\infty(M\times S^1).
\end{equation}
\begin{cor}\label{cor:heavy_sets}
	Let $(M,\omega)$ be a monotone symplectic manifold with monotonicity constant $\kappa$. Let $U\subset M$ be a portable Liouville domain with a dynamically convex incompressible boundary, and assume in addition that either $\kappa\leq 0$ or $C(U)\leq \kappa$. Then, $M\setminus U$ is superheavy.
\end{cor}

\subsection*{Organization of the paper.}
Section~\ref{sec:preliminaries} contains an overview of the necessary preliminaries and fixes some notations. In Section~\ref{sec:killers}, we construct ``spectral killers" for Hamiltonians supported in a domain with incompressible contact type boundary, under a certain condition involving the Reeb dynamics on the boundary and the spectral invariant of the Hamiltonian. We also explain how the existence of spectral killers implies the max inequality and prove Theorem~\ref{thm:max_ineq_rational_extendable}. In Section~\ref{sec:proving_max_ineq} we show that the aforementioned condition  holds in various settings, and thus prove Theorems~\ref{thm:max_ineq_rational}, \ref{thm:max_ineq_neg_monotone} and \ref{thm:max_ineq_pos_monotone}. Section~\ref{sec:unif_bounds} concerns uniform bounds on spectral invariants and contains the proofs of Theorems~\ref{thm:spec_bound_pos} and \ref{thm:spec_bound_neg}. Finally, in Section~\ref{sec:the_invariant_C} we estimate the invariant $C(U)$ on certain classes of domains.

\subsection*{Acknowledgements.} 
{I am very grateful to my advisors Lev Buhovsky and Leonid Polterovich for their guidance and insightful inputs. I also thank Yaniv Ganor, Vincent Humili\`ere, R\'emi Leclercq and Sobhan Seyfaddini for useful discussions. 
It was recently brought to my attention that some of the arguments in this paper were known to Matthew Strom Borman, including the idea to use spectral killers in order to prove a max inequality. I am grateful to him for kindly sharing his previous findings and ideas with me. 
The research leading to these results was partially funded by the Israel Science Foundation, grants 1102/20 and 2026/17, as well as by the Levtzion Scholarship.}

\section{Preliminaries.}\label{sec:preliminaries}
Let us review the necessary preliminaries and fix some notations. Note that throughout the paper we assume $(M,\omega)$ to be a rational closed symplectic manifold, namely, $\omega(\pi_2(M))$ is a discrete subgroup of $\Z$. 

\subsection{Hamiltonian Floer homology.}
Given a Hamiltonian $H:M\times S^1\rightarrow\R$, its symplectic gradient is the vector field defined by the equation $\omega(X_H, \cdot) = -dH$ and the flow $\varphi_H^t $ of this vector field is called the Hamiltonian flow of $H$.
The set of 1-periodic orbits of $\varphi_H^t$ is denoted by $\cP(H)$. The Hamiltonian $H$ is called {\it non-degenerate} if the graph of $d\varphi_H^1$ is transversal to the diagonal in $TM\times TM$. Equivalently, $H$ is non-degenerate if every $\gamma\in \cP(H)$ is non-degenerate, that is, if 1 is not an eigenvalue of $d\varphi_H^1(\gamma(0))$ for every $\gamma\in \cP(H)$.

We denote by  $\cL M$ the space of contractible loops in $M$. A capping disk of $\gamma\in\cL M$ is a map $u:D\rightarrow M$ from the unit disk to $M$, satisfying $u|_{\partial D}=\gamma$. Two capping disks $u_1, u_2$ of $\gamma$ are equivalent if $[u_1\#(-u_2)]\in \ker \omega\cap \ker c_1$. We denote by $\widetilde{\cL M}$ the space of equivalence classes of capped loops, $(\gamma,u)$. 
The action functional corresponding to $H$ is defined on the space $\widetilde{\cL M}$ by
\begin{equation*}
\cA_H(\gamma,u) = \int_0^1 H(\gamma(t),t)\ dt -\int_u \omega.
\end{equation*}
The critical points of the action functional are (equivalence classes of) capped 1-periodic orbits of $\varphi_H^t$ and the set of their values is denoted by $\spec(H)$. For a non-degenerate Hamiltonian $H$ and a generic $\omega$-compatible almost complex structure $J$, the Floer chain complex $CF_*(H,J)$ is generated by these critical points and its differential is defined by counting certain negative gradient flow lines of $\cA_H$ (with respect to a metric induced by $J$ on $\widetilde{\cL M}$). 
For more details, see,  e.g., \cite{mcduff2012j,polterovich2014function}.
The chain complex $CF_*(H,J)$ is graded by the Conley-Zehnder (abbreviated to CZ) index, whose definition is recalled below. More formally, $CF_k(H,J)$ is generated by (equivalence classes of) capped 1-periodic orbits whose index\footnote{We add the subscript $H$ to the notation in order to distinguish this index from the contact CZ index, which will be discussed later.},  $\indCZ_H(\gamma,u)$, is equal to $-k$. For $k\in\Z$, we denote by $\spec_k(H)$ the set of action values of the generators of $CF_k(H,J)$.

\subsubsection{The Conley-Zehnder and Robbin-Salamon indices.}
The Conley-Zehnder index is defined for non-degenerate capped 1-periodic orbits, through an index of a path of symplectic matrices which is obtained from the linearized flow after trivializing the tangent bundle. In \cite{robbin1993maslov}, Robbin and Salamon defined a Maslov-type index for possibly degenerate orbits, that coincides with the Conley-Zehnder index on non-degenerate ones. Since we will consider degenerate Hamiltonians as well, we give here the definition of the Robbin-Salamon index,  following the exposition in \cite{gutt2014generalized}. 
\begin{itemize}
	\item Let $\Phi =\{\Phi(t)\}_{t\in[0,T]}\subset\Sp(2n)$ be a path of symplectic matrices. A number $t\in[0,T]$ is called a {\it crossing} if $\det(\Phi(t)-\id)=0$. The {\it crossing form} $\Gamma_t=\Gamma_t(\Phi)$	is the quadratic form obtained by restricting the symmetric matrix $S(t):=-J_0\dot{\Phi}(t)\Phi^{-1}(t)$ to $\ker(\Phi(t)-\id)$. A crossing $t_0$ is called {\it regular} if the crossing form $\Gamma_{t_0}$ is non-degenerate.
	
	\item For a path $\Phi =\{\Phi(t)\}_{t\in[0,T]}$ having only regular crossings, the {\it Robbin-Salamon index} (abbreviated to RS) is defined to be 
	\begin{equation}\label{eq:RS_ind}
	\indRS(\Phi):=\frac{1}{2}\sign(\Gamma_0)+\sum_{0<t<T} \sign(\Gamma_t)+\frac{1}{2}\sign(\Gamma_T),
	\end{equation}
	where the sum is taken over all crossings $t\in(0,T)$. 
	The definition of the RS index for general paths of symplectic matrices is given in \cite{gutt2014generalized}. This index admits several useful properties:
	\begin{itemize}
		\item (homotopy) The RS index is invariant under homotopies with fixed ends.
		\item (products) $\indRS(\Phi_1\oplus\Phi_2) = \indRS(\Phi_1)+ \indRS(\Phi_2)$.
		\item (conjugation) $\indRS(\Psi\Phi\Psi^{-1})=\indRS(\Phi)$ for every path $\Psi$ of symplectic matrices.
		\item (concatenation) $\indRS(\Phi|_{[t_1,t_3]}) = \indRS(\Phi|_{[t_1,t_2]})+\indRS(\Phi|_{[t_2,t_3]})$.
		\item (inverse) $\indRS(\bar\Phi)=\indRS(\Phi^T)=-\indRS(\Phi)$, where $\bar\Phi(t):=\Phi(-t)$ and $\Phi^T$ is the path of transposed matrices.
	\end{itemize}
	Given a general path we may perturb it with fixed ends to obtain a path with only regular crossings, and use (\ref{eq:RS_ind}) to compute the index.
	For a path $\Phi$  with only regular crossings such that $\Phi(0)=\id$, and $t=1$ is not a crossing, the RS index coincides with the Conley-Zehnder  index of the path.
	
	\item To define the index of a capped periodic orbit $(\gamma,u)$, consider a trivialization of the tangent bundle $TM$ along $u$. Then, the differential of the flow $\varphi_H^t$ along the loop $\gamma$ is identified with a path of symplectic matrices, $d\varphi_H^t(\gamma(0))\mapsto\Phi(t)\in\Sp(2n)$. The RS index of the capped orbit is given by $\indRS(\gamma,u):=\indRS(\Phi)$.
	Throughout the paper, we use the notation $\indCZ$ for non-degenerate orbits and $\indRS$ for degenerate ones. 
	A useful property of the CZ index is that the indices with respect to different capping disks differ by twice the first Chern class of the connected sum. More formally, let $\gamma\in\cP(H)$ be a non-degenerate 1-periodic orbit of a Hamiltonian $H$ and let $u,v:D\rightarrow M$ be two different capping disks for $\gamma$. Then, for $A:= u\# (-v)\in\pi_2(M)$, 
	\begin{equation}
	\indCZ_H(\gamma,u) =\indCZ_H(\gamma,v\#A) =\indCZ_H(\gamma,v)+2c_1(A),
	\end{equation}
	where $c_1$ denotes the first Chern class of $M$, see, e.g., \cite[Sections 2.6-7]{mcduff2012j}.
\end{itemize}
We remark that there are texts choosing an opposite sign for the CZ index, such as \cite{seyfaddini2014spectral}. In our sign convention, the index of a critical point $p$ of a $C^2$-small Morse function with a constant capping disk, $u_p(D)=\{p\}$, is related to the Morse index via $\indCZ_H(p,u_p)=n-i_{Morse}(p)$.
\subsection{Spectral invariants.}
The Floer complex admits a natural filtration by the action value. Let $CF_*^a(H,J)$ be the sub-complex generated by (equivalence classes of) capped 1-periodic orbits whose action is bounded by $a$ from above. Since the Floer differential is action decreasing, it restricts to the sub-complex $CF_*^a(H,J)$ and the  homology $HF_*^a(H,J)$ is well defined. 
The spectral invariant with respect to the fundamental class is defined to be the smallest value of $a$ for which the fundamental class appears in $HF^a_*(H,J)$, namely,
\begin{equation}
c(H) := \inf\{a:[M]\in \im(\iota^a_*)\},
\end{equation}
where $\iota^a_*:HF_*^a(H,J)\rightarrow HF_*(H,J)$ is the map induced by the inclusion $\iota^a:CF_*^a(H,J)\hookrightarrow CF_*(H,J)$.
We remark that spectral invariants are defined for general quantum homology classes, but we consider only the spectral invariant with respect to the fundamental class.
Spectral invariants have several useful properties, let us state the relevant ones:  
\begin{itemize}
	\item (stability)  For any Hamiltonians $H$ and $G$,
	\begin{eqnarray} 
	\nonumber\int_{0}^{1} \min_{x\in M}(H(x,t)-G(x,t))dt \leq c(H) - c(G) \leq  \int_{0}^{1} \max_{x\in M}(H(x,t)-G(x,t))dt.
	\end{eqnarray} 
	In particular, $c:\cC^\infty(M\times S^1)\rightarrow\R$ is a continuous functional and is extended by continuity to degenerate Hamiltonians. Moreover, this implies that the spectral invariant is monotone: If $G(x,t)\leq H(x,t)$ for all $(x,t)\in M\times S^1$, then $c(G)\leq c(H)$.
	
	\item (spectrality) $c(H)\in\spec(H)$. Moreover, if $H$ is non-degenerate, $c(H) \in \spec_n(H)$.
	
	\item (subadditivity) For every Hamiltonians $H$ and $G$, one has $c(H\#G)\leq c(H)+c(G)$, where $H\# G:= H+G\circ(\varphi_H^t)^{-1}$. In particular, if $H$ and $G$ are disjointly supported then $c(H+G)\leq c(H)+c(G)$. 
	\item(energy-capacity inequality) If the support of $H$ is displaceable, its spectral invariant is bounded by the displacement energy of the support, namely, $c(H)\leq e(supp(H))$.
	We remind that a subset $X\subset M$ is displaceable if there exists a Hamiltonian $G$ such that $\varphi_G^1(X)\cap X=\emptyset$. In this case, the displacement energy of $X$ is given by 
	\begin{equation}\label{eq:disp_energy_def}
	e(X):=\inf_{G: \varphi_G^1(X)\cap X=\emptyset} \int_0^1 \left(\max_M G(\cdot,t)-\min_M G(\cdot,t)\right)\ dt. 
	\end{equation}
\end{itemize}
For a wider exposition see, for example, \cite{mcduff2012j,polterovich2014function}.
A standard method for estimating the spectral invariant of a Hamiltonian $H$ is through a {\it bifurcation diagram}. Given a continuous deformation $\{H_\tau\}_\tau$ of $H$, the corresponding bifurcation diagram is the set $\cup_\tau \left(\{\tau\}\times\spec(H_\tau)\right)\subset \R^2$. By the spectrality and stability properties, the spectral invariant of $H_\tau$ moves continuously in the diagram as $\tau$ varies. Therefore, if we construct a deformation such that the value of the spectral invariant is known at a certain point of the deformation, we can study the bifurcation diagram in order to estimate $c(H)$. This approach was used by Polterovich in \cite{polterovich2014symplectic}, by Seyfaddini in \cite{seyfaddini2014spectral} and by Ishikawa in \cite{ishikawa2015spectral}, and is frequently used in the present paper as well.

\subsection{Contact hyper-surfaces, the Reeb flow and Liouville domains.}\label{subsec:preliminaries_Liouville_coord_and_Reeb_flow}
As mentioned earlier, a domain $U\subset M$ has a {\it contact type boundary} if there exists a vector field $Y$, called {\it the Liouville vector field}, that is defined on a neighborhood of the boundary $\partial U$, satisfies $\cL_Y\omega=\omega$, is transverse to the boundary and points outwards. The flow $\psi^s$ of $Y$ defines a radial coordinate, called the {\it Liouville coordinate}, on a tubular neighborhood of the boundary, by  
\begin{equation}\label{eq:Lioville_coordinate}
\partial U\times(1-\epsilon,1+\epsilon)\cong \cN(\partial U),\ \ (y,s)\mapsto x=\psi^{\log(s)} y.
\end{equation}
Note that the Liouville flow expands the symplectic form, namely $\left(\psi^{\log(s)}\right)^*\omega=s\cdot \omega$.
The 1-form  $\lambda:=\iota_Y\omega$ is a primitive of $\omega$ and the kernel of its restriction to $T\partial U$, namely $\xi:=\ker\lambda|_{T\partial U}$, is called the {\it contact distribution}. We denote by $R$ the {\it Reeb vector field}, which is defined on {a neighborhood of} 
$\partial U$ by 
\begin{equation}\label{eq:Reeb_def}
\lambda(R)=1,\ R_{(y,s)}\in\ker d\lambda_{(y,s)}|_{T\psi^{\log(s)}\partial U}.
\end{equation}
We remark that the Liouville vector field is not unique (and hence so are the 1-form $\lambda$ and the Reeb vector field). In fact, for every $C^1$-small Hamiltonian $H$ defined near $\partial U$, $Y':=Y-X_H$ is also a Liouville vector field. 
If the vector filed $Y$ extends to $U$ (and satisfies $\cL_Y\omega=\omega$ there), we say that $U$ is a {\it Liouville domain}. 
 
Denote by $\varphi_R^t:\partial U\rightarrow \partial U$ the flow of the Reeb vector field, and let $\cP(\partial U)$ be the set of all contractible periodic orbits of $\varphi_R^t$ (of any period). The action of such orbits is defined to be the integral of the 1-form $\lambda$ along the orbit, and coincides with the period. The set of action values is called {\it the Reeb spectrum} of $\partial U$ and is denoted by $\spec(\partial U)$. We say that the Reeb flow is  {\it non-degenerate} if the graph of the restriction  of $d\varphi_R^t$ to the contact distribution $\xi$ intersects the diagonal in $\xi\times \xi$ transversely. In this case, for each orbit $\gamma \in \cP(\partial U)$ and a capping disk $u\subset \partial U$, one can assign an integer, which we call the contact Conley-Zehnder index and denote\footnote{We add the subscript $R$ to the notation in order to distinguish the contact CZ index from the Hamiltonian CZ index that appeared earlier.} by $\indCZ_R(\gamma,u)$, in the following way. Trivializing the contact distribution $\xi$ along the disk $u$, the restriction of the linearized flow $d\varphi_R^t$ to $\xi$ along $\gamma$ is identified with a path of symplectic matrices, $d\varphi_R^t|_{\xi}(\gamma(0))\mapsto\Phi(t)\in \Sp(2n-2)$, see e.g., \cite{bourgeois2009survey}. The index $\indCZ_R(\gamma,u)$ is defined to be the RS index of the path $\Phi$, as in (\ref{eq:RS_ind}).

\subsubsection{Portable Liouville domains.} \label{subsec:portable_preliminaries}
As mentioned earlier, a Liouville domain $U$ is called portable if its core, which is given by $Q:=\cap_{s\in(0,1]} \psi^{\log(s)}U$, is displaceable in $U$. In this case, the displacement energy of $\psi^{\log(s)}U$ is arbitrarily small as $s$ approaches zero, as explained in \cite[p.499]{polterovich2014symplectic}. This fact will be used in the proof of Theorem~\ref{thm:spec_bound_pos} which asserts a uniform bound for spectral invariants of Hamiltonians supported in portable Liouville domains with incompressible dynamically convex boundaries on positively monotone manifolds, see Section~\ref{subsec:bound_pos_mon}.

\section{Constructing spectral killers.}\label{sec:killers}
In \cite{seyfaddini2014spectral}, Seyfaddini presented a construction of certain functions, called {\it spectral killers}, which can be used to produce upper bounds for the spectral invariant of a sum of disjointly supported Hamiltonians. A spectral killer for a Hamiltonian $H$ supported in a domain $U$ is a function $K:M\rightarrow\R$ supported in $U$ such that $c(H+K)=0$. 
\begin{claim}[Seyfaddini]\label{clm:max_ineq}
	Let $H_1,\dots, H_N$ be Hamiltonians supported in pairwise disjoint domains $U_1,\dots,U_N\subset M$ and suppose there exist Hamiltonians $K_i$ supported in $U_i$, such that $c(H_i+K_i)=0$. Then,
	\begin{equation}
	c(H_1+\cdots+H_N)\leq \max_i\|K_i\|_{C^0}.
	\end{equation} 
\end{claim}
\begin{proof}
	The following argument is taken from \cite{seyfaddini2014spectral}. Using the stability and subadditivity properties of spectral invariants and noticing that $\{H_i+K_i\}$ are all disjointly supported, we have
	\begin{eqnarray*}
		c(H_1+\cdots+H_N) 
		&\leq& c\left(\sum_i (H_i+K_i)\right) +\Big\|-\sum_i K_i\Big\|_{C^0}\\
		&\leq& \sum_i c(H_i+K_i) +\Big\|\sum_i K_i\Big\|_{C^0}= \sum_i 0 +\max_i \|K_i\|_{C^0}.
	\end{eqnarray*}
\end{proof} 

In \cite{seyfaddini2014spectral}, Seyfaddini constructed spectral killers for  Hamiltonians supported in displaceable balls having small displacement energies in monotone manifolds. The norms of the spectral killers are bounded by the capacities of the balls.

In order to obtain the max inequality (\ref{eq:max_ineq}) we will construct spectral killers whose norms are equal to the spectral invariants of the Hamiltonians $\{H_i\}$ whenever $c(H_i)>0$. Note that for Hamiltonians with non-positive spectral invariants the max inequality follows from the subadditivity property of spectral invariants.   
In this section we show that under a certain condition on the domain $U$ and the spectral invariant $c(H)$ of $H$, there exists such a spectral killer $K$. We begin with a simpler construction of what we call a ``slow spectral killer", which is supported on a domain larger than $U$.

\subsection{A simpler case: ``slow spectral killers".}
In this section we use the notion of $\sigma$-extendable domains from Definition~\ref{def:extendable}. Our goal is to prove the following statement.
\begin{prop}\label{pro:slow_killers}
	Let $(M,\omega)$ be a rational symplectic manifold, namely $\omega(\pi_2(M))=\kappa\cdot \Z$, and let $U\subset M$ be a $\sigma$-extendable domain with an incompressible contact type boundary. Then, for every Hamiltonian $H$ supported in $U$ such that 
	\begin{equation}\label{eq:assumption_slow_killer}
	0\leq c(H)<\min\{\kappa,\sigma\cdot\min \spec(\partial U)\}
	\end{equation} 
	there exists $K:M\rightarrow \R$ supported in $(1+\sigma)U$ with $\|K\|_{C^0}=c(H)$ such that $c(H+K)=0$.
\end{prop}
The above proposition guarantees that the max inequality (\ref{eq:max_ineq}) holds for Hamiltonians supported in disjoint extendable domains that satisfy (\ref{eq:assumption_slow_killer}), only if we assume in addition that the supports  $(1+\sigma_i)U_i$  of the spectral killers are disjoint. In particular, Theorem~\ref{thm:max_ineq_rational_extendable} is an immediate consequence of Proposition~\ref{pro:slow_killers} and Claim~\ref{clm:max_ineq}:
\begin{proof}[Proof of Theorem~\ref{thm:max_ineq_rational_extendable}]
	Let $\{H_i\}$ be Hamiltonians supported in domains $\{U_i\}$ with incompressible contact type boundaries such that $U_i$ is $\sigma_i$-extendable, the sets $\{(1+\sigma_i)U_i\}$ are pairwise disjoint, and $c(H_i)<\min\{\kappa,\sigma_i\cdot\min\spec(\partial U_i)\}$.  By Proposition~\ref{pro:slow_killers} there exist $K_i:M\rightarrow\R$ supported in $(1+\sigma_i)U_i$ such that $\|K_i\|_{C^0}=c(H_i)$ and $c(H_i+K_i)=0$. Applying Claim~\ref{clm:max_ineq} to $\{H_i\}$ and $\{K_i\}$ with the disjoint domains $\{(1+\sigma_i)U_i\}$, we conclude that the max inequality (\ref{eq:max_ineq}) holds for $\{H_i\}$.
\end{proof}

In order to prove  Proposition~\ref{pro:slow_killers} we need some preliminary notations and calculations. 
We use the notations of Section~\ref{sec:preliminaries} and, in particular, the Liouville coordinate $s$ defined in (\ref{eq:Lioville_coordinate}), using the Liouville vector field.
The spectral killer $K$ will be a function of the Liouville coordinate. This will enable us to relate its 1-periodic orbits to the Reeb orbits on $\partial U$. 
\begin{defin}\label{def:radial_hamiltonian}
	We say that an autonomous Hamiltonian $H$  is {\it radial} if there exists $\chi:\R\rightarrow\R$ such that $H=\chi(s)$ wherever the the Liouville coordinate is defined, and is locally constant elsewhere.
\end{defin} 
Radial Hamiltonians and their periodic orbits were studied in \cite{hofer1998dynamics,ishikawa2015spectral} for the case where $U$ is a strictly convex domain in $\R^{2n}$.	
The following lemma relates between the 1-periodic orbits of radial Hamiltonians and the Reeb orbits on $\partial U$. 
\begin{lemma}\label{lem:radial_Hamiltonians1}
	Let $H=\chi(s)$ be a radial Hamiltonian, then the Hamiltonian flow of $H$ is conjugated to the Reeb flow up to a time reparametrization:
	\begin{equation}\label{eq:Ham_Reeb_flow_conj}
	\varphi_H^t=\psi^{\log s}\circ \varphi_R^{\chi'(s)\cdot t}\circ\psi^{-\log s}.
	\end{equation}
	In particular, every non-constant 1-periodic orbit $\gamma$  of $H$ is contained in a level set of the Liouville coordinate $s=s(\gamma)$ and is conjugated to a periodic Reeb orbit, $\hat{\gamma}\in\cP(\partial U)$, via $\hat\gamma(\chi'(s)\cdot t) = \psi^{-\log s}\gamma(t)$. Moreover, the Reeb action of $\hat \gamma$ is equal to the absolute value of the derivative of $\chi$ at $s=s(\gamma)$:
	\begin{equation}\label{eq:action_vs_dchi}
	\int_{\hat\gamma}\lambda =|\chi'(s)|.
	\end{equation}
	In particular, $|\chi'(s)|$ belongs to the Reeb spectrum of $\partial U$.
\end{lemma}
\begin{proof}[Proof of Lemma~\ref{lem:radial_Hamiltonians1}]
	Let us prove the following relation between the Hamiltonian vector field $X_H$ of $H$ and the Reeb vector field:
	\begin{equation}\label{eq:Ham_Reeb_vf_conj}
	d\psi^{-\log s} X_H\circ\psi^{\log s}=\chi'(s)\cdot R.
	\end{equation}
	Note that (\ref{eq:Ham_Reeb_flow_conj}) will follow from uniqueness of solutions of ODEs. Let us show that the LHS of (\ref{eq:Ham_Reeb_vf_conj}) satisfies the equations defining the Reeb vector field with a factor of $\chi'(s)$. Given any vector $v\in T_xM$ for  $x\in\partial U$,
	\begin{eqnarray}
	\omega_x(v, d\psi^{-\log s} X_H\circ\psi^{\log s}(x)) &=& (\psi^{-\log s}_x)^*\omega_{\psi^{\log s} x}(d\psi^{\log s} v, X_H) \nonumber\\
	&=&s^{-1}\cdot \omega_{\psi^{\log s} x}(d\psi^{\log s} v, X_H) \nonumber\\
	&=& s^{-1} dH_{\psi^{\log s} x}(d\psi^{\log s} v).\label{eq:intermediate_XH_vs_R}
	\end{eqnarray} 
	When  $v\in T\partial U$, its image under the linearized Liouville flow, $d\psi^{\log s} v$, is tangent to a level set of the Liouville coordinate $s$ and hence to a level set of $H$. In this case,
	\begin{equation*}
	\omega_x(v, d\psi^{-\log s} X_H\circ\psi^{\log s}(x)) = s^{-1}dH_{\psi^{\log s} x}(d\psi^{\log s} v) =0
	\end{equation*}
	On the other hand, taking $v$ to be the Liouville vector field, equation  (\ref{eq:intermediate_XH_vs_R}) implies:
	\begin{eqnarray*}
		\omega_x(Y, d\psi^{-\log s} X_H\circ\psi^{\log s}(x)) &=& s^{-1} dH_{\psi^{\log s} x}(d\psi^{\log s} Y)\\
		&=& s^{-1}\cdot\frac{d}{d\tau}\Big|_{\tau = 0} H\circ \psi^{\log (s)}  \psi^\tau(x)\\
		&=& s^{-1}\cdot\frac{d}{d\tau}\Big|_{\tau = 0} H\circ \psi^{\log (s\cdot e^\tau)}  (x)\\
		&=& s^{-1}\cdot\frac{d}{d\tau}\Big|_{\tau = 0} \chi(s\cdot e^\tau) = \chi'(s).
	\end{eqnarray*}
	
	Having established the conjugation of the Hamiltonian and the Reeb flows (\ref{eq:Ham_Reeb_flow_conj}), we turn to prove the relation between the Reeb action and the derivative of $\chi$ stated in equation (\ref{eq:action_vs_dchi}). For a 1-periodic orbit $\gamma$ of $H$ lying in level $s$, let $\hat{\gamma}\subset \partial U$ be the curve defined by $\hat \gamma(\chi'(s)\cdot t) =\psi^{-\log s}\gamma(t)$. Then, using (\ref{eq:Ham_Reeb_flow_conj}) we find
	\begin{eqnarray}
	\frac{d}{dt}\hat\gamma(t) &=& \frac{d}{dt}\left(\psi^{-\log s}\gamma(t/\chi'(s))\right) = \frac{d}{dt}\left(\psi^{-\log s}\varphi_H^{t/\chi'(s)}\gamma(0)\right) \nonumber\\
	&=& \frac{d}{dt}\left(\varphi_R^{t\cdot \chi'(s)/\chi'(s)}\psi^{-\log s}\gamma(0)\right)=R\nonumber.	
	\end{eqnarray} 
	We conclude that $\hat \gamma$ is a periodic Reeb orbit whose period, and therefore action, is equal to $|\chi'(s)|$.
\end{proof}
We are now ready to prove Proposition~\ref{pro:slow_killers}.
\begin{proof}[Proof of Proposition~\ref{pro:slow_killers}]
	Let $U$ be a $\sigma$-extendable domain with an incompressible contact type boundary and let $H$ be a Hamiltonian supported in $U$. 
	Let $\cN(\partial U)\cong\partial U\times(1-2\epsilon, 1+\sigma)$ be a tubular neighborhood of the boundary on which the Liouville coordinate is defined, and assume that $\epsilon$ is small enough so that $H|_{\cN(\partial U)}=0$.
	\begin{figure}
		\centering
		\includegraphics[scale=0.6]{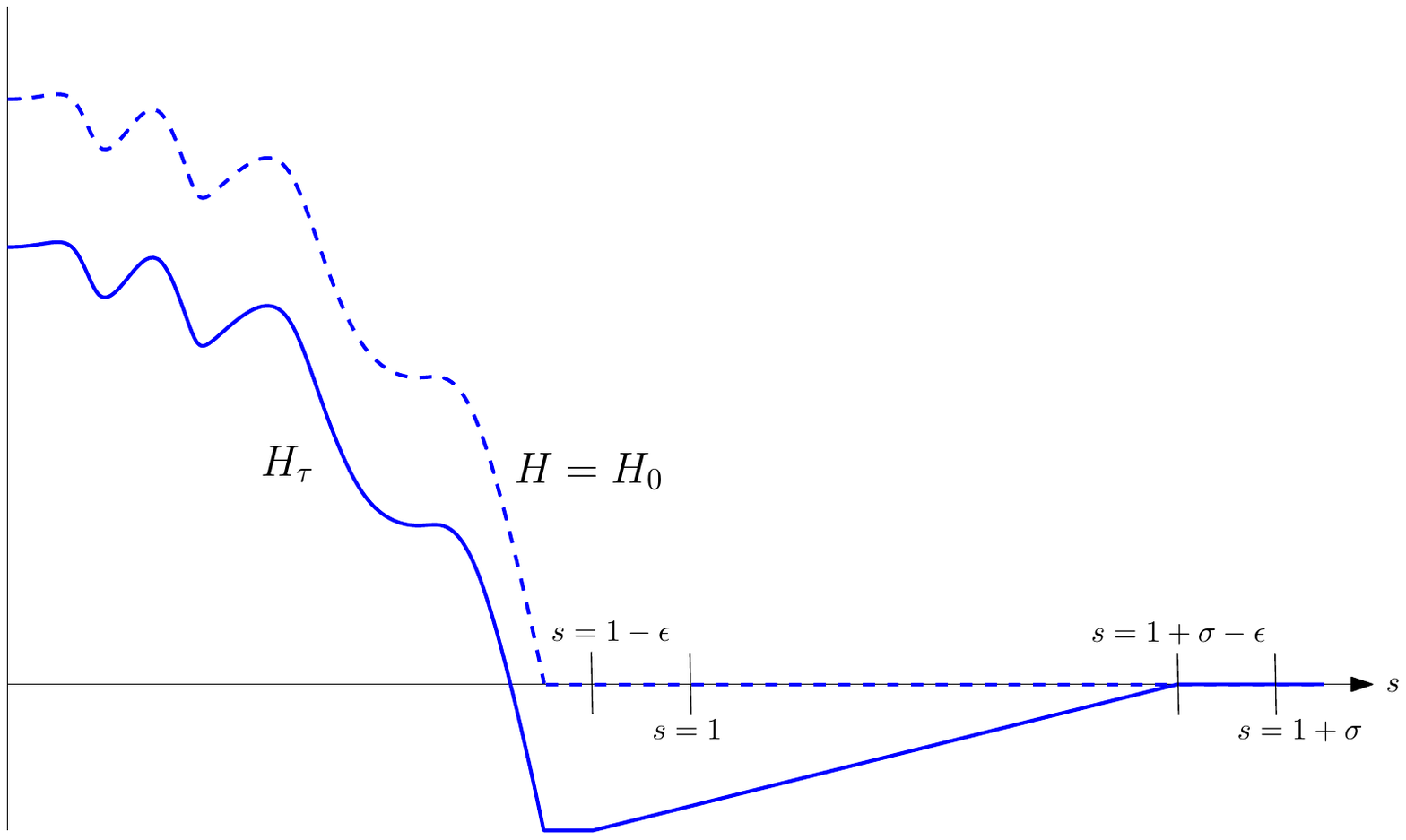}
		\caption{\small{An illustration of the graphs of $H_\tau$ (solid line) and $H$ (dashed line) in the radial coordinate $s$. 
		}}
		\label{fig:deformation_slow_killer}
	\end{figure}
	\noindent
	Consider the  autonomous radial Hamiltonian defined by
	\begin{equation}
	K_1(x) :=\begin{cases}
	-1 & x\in U\setminus\big(\partial U\times (1-\epsilon,1)\big),\\
	\chi_1(s(x)) & x\in\partial U\times(1-\epsilon,1+\sigma-\epsilon),\\
	0 & \text{elsewhere}
	\end{cases}
	\end{equation}
	where $\chi_1:\R\rightarrow\R$ is a smooth approximation of the continuous piecewise linear function taking the value $-1$ for $s\leq 1-\epsilon$ and $0$ for $s\geq1+\sigma -\epsilon$. We choose $\chi_1$ such that its derivative is bounded by $1/\sigma$ and that, for $s>1-\epsilon$, its derivative vanishes only in the ``flat region", i.e. where $\chi_1$ itself vanishes. For each $\tau\in[0,c(H)]$, set $K_\tau:= \tau\cdot K_1$, then $\{K_\tau\}_\tau$ is a continuous family of radial Hamiltonians. Denoting $\chi_\tau:=\tau\cdot\chi_1$,   its derivative is bounded by $\tau/\sigma$, and therefore $K_\tau$ has no non-constant 1-periodic orbits as long as $\tau/\sigma<\min \spec(\partial U)$. This follows from Lemma~\ref{lem:radial_Hamiltonians1}, which states that non-constant 1-periodic orbits of radial Hamiltonians appear only for $s$ such that $|\chi'(s)|\in\spec(\partial U)$.  Assumption (\ref{eq:assumption_slow_killer}) in Proposition~\ref{pro:slow_killers} states that $c(H)<\sigma\cdot\min \spec(\partial U)$ and therefore $K_{\tau}$ has no non-constant 1-periodic orbits for all $\tau\in[0,c(H)]$.
	We will show that $K_{c(H)}$, which is supported in $(1+\sigma)U$, is a spectral killer for $H$, namely, that $c(H+K_{c(H)})=0$.
	Consider the deformation of the Hamiltonian $H$ given by $\{H_\tau = H+K_\tau\}_{\tau\in[0,c(H)]}$, which is illustrated in Figure~\ref{fig:deformation_slow_killer}, and let us study the corresponding bifurcation diagram.
	We remind that by the stability and spectrality properties, the spectral invariant moves continuously in this diagram. Let us show that the bifurcation diagram consists of lines with slope $-1$, corresponding to orbits in $U$, and of horizontal lines with values in $\kappa\Z$, corresponding to constant orbits outside of $U$, as illustrated in  Figure~\ref{fig:bifurcation_diag_slow_killer}. Indeed:
	\begin{figure}
		\centering
		\includegraphics[scale=0.5]{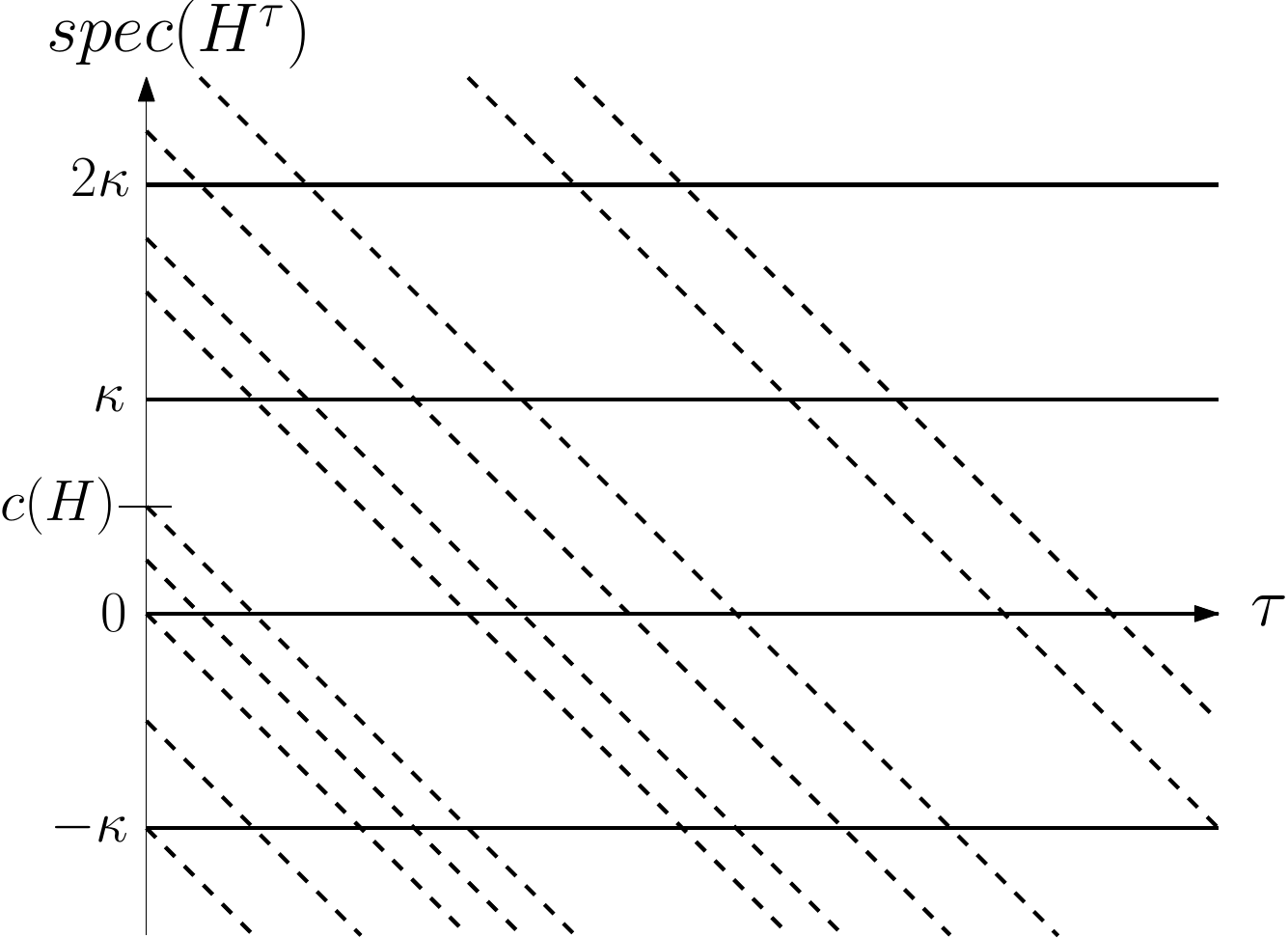}
		\caption{\small{An illustration of the $n$-spectrum of a non-degenerate perturbation of $H_\tau$. The dashed lines correspond to actions of orbits in $\{s<1-\epsilon/2,1\}$ while the solid lines correspond to actions of  orbits in $\{s>1-\epsilon/2\}$.}}
		\label{fig:bifurcation_diag_slow_killer}
	\end{figure}\noindent
	\begin{itemize}
		\item Orbits in $\{s> 1-\epsilon\}$: In this region $H_\tau$ coincides with $K_\tau$ which has only constant 1-periodic orbits. Their actions are given by $\chi_\tau(s) -\omega(A)$ for some $A\in \pi_2(M)$ and $s$ such that $\chi_\tau'(s)=0$. When $s>1-\epsilon$, the derivative of $\chi_\tau$ vanishes if and only if $s>1+\sigma-\epsilon$, in which case $\chi_\tau(s)=0$. Overall, the 1-periodic orbits of $H_\tau$ in this region all have actions lying in $\kappa \Z$ and do not change with $\tau$. 
		\item Orbits in $\{s\leq 1-\epsilon\}$: Here $H_\tau=H-\tau$. Therefore, the 1-periodic orbits of $H_\tau$ do not change with $\tau$ but  their actions decrease linearly: 
		$$
		\cA_{H_\tau}(\gamma,\tau)=\int_0^1 H_\tau (\gamma(t),t)\ dt -\int_u\omega =   \int_0^1 H (\gamma(t),t)\ dt -\tau -\int_u\omega .
		$$
	\end{itemize}	
	We conclude that the bifurcation diagram consists of decreasing lines of slope $-1$ and of horizontal lines with values in $\kappa\Z$. By assumption (\ref{eq:assumption_slow_killer}), $c(H)<\kappa$ and thus there are no horizontal lines between $0$ and $c(H)$. It follows that the spectral invariant $c(H_\tau)$ moves along a single decreasing line and hence $c(H+K_\tau)=c(H_\tau)=c(H)-\tau$ for $\tau\in[0,c(H)]$. In particular, $c(H+K_{c(H)})=0$ and $K_{c(H)}$ is the required spectral killer.
\end{proof}

\subsection{Spectral killers that are supported in $U$.}
In order to construct spectral killers that are supported in $U$, rather than on a larger domain, we must consider radial Hamiltonians with arbitrarily large slopes. Such Hamiltonians may have a lot of non-constant 1-periodic orbits which correspond to Reeb orbits on the boundary of $U$, as stated in Lemma~\ref{lem:radial_Hamiltonians1}. In this case, we need to impose certain  assumptions on the spectral invariant of the Hamiltonian considered and the Reeb dynamics on the boundary. 
Let us fix a domain $U$ with an incompressible contact type boundary such that the Reeb flow is non-degenerate. The non-degeneracy of the Reeb flow can be achieved by  a small perturbation of the Liouville vector field (or, equivalently, the contact form), see, e.g., \cite[p.47]{bourgeois2009survey}.  For a non degenerate domain we define the {\it  relative $n$-spectrum} of the boundary $\partial U$ in $M$:
\begin{defin}
	\begin{itemize}
	\item
	For $k\in\Z$, we denote by
	\begin{equation*}
	\spec_k(\partial U):=\left\{\int_\gamma\lambda\ :\ \begin{array}{c}
	\gamma\in\cP(\partial U),\\
	\exists u:D\rightarrow \partial U,\ 
	u|_{\partial D}=\gamma,\ 
	\indCZ_R(\gamma,u)=-k
	\end{array}\right\}
	\end{equation*}
	the action spectrum of $\partial U$ of index $-k$.
	\item
	Denote
	\begin{equation*}
	\spec_n(0):= \Big\{- \omega(A):A\in \pi_2(M),\ c_1(A)\in\{-n,\dots,0\}\Big\}.
	\end{equation*}
	Morally speaking, $\spec_n(0)$ is the ``$n$-spectrum of the zero function", namely, the set actions that can be attained by index $-n$ capped orbits of a non-degenerate perturbation of the zero function.
	
	\item
	We define the relative $n$-spectrum of $\partial U$ inside $M$ to be
	\begin{equation*}
	\spec_n(\partial U;M):=\spec_n(0)\cup  
	\bigcup_{\tiny{\begin{array}{c}
		{k\in\Z}, A\in\pi_2(M),\\
		c_1(A)=\lceil \frac{k-n}{2}\rceil 
		\end{array}}}
	\left\{-\spec_k(\partial U)-\omega(A)\right\}.
	\end{equation*}
	\end{itemize}
\end{defin}

\noindent We will show that the relative $n$-spectrum $\spec_n(\partial U;M)$ is related to the $n$-spectrum of radial Hamiltonians (see Definition~\ref{def:radial_hamiltonian}).
The sets $\spec_k(\partial U)$ and  $\spec_n(\partial U;M)$ depend on the choice of a Liouville vector field (or, equivalently, a contact form). We omit the Liouville vector field from the notation for the sake of brevity.
Our main goal for this section is to prove the following statement.
\begin{prop}\label{pro:killer1}
	Let $H:M\times S^1\rightarrow\R$ be a Hamiltonian supported in a domain $U$ with an incompressible contact type boundary such that the Reeb flow is non-degenerate. Assume that $c(H)>0$ and 
	\begin{equation}\label{eq:the_condition}
	\spec_n(\partial U;M)\cap (0,c(H)]=\emptyset.
	\end{equation}
	Then, there exists a Hamiltonian $K:M\rightarrow\R$ supported in $U$ such that $\|K\|_{C^0}= c(H)$ and $c(H+K)=0$.
\end{prop}
\begin{rem}\label{rem:killer_for_degenerate}
	The non-degeneracy of the Reeb flow on $\partial U$ is required for the contact CZ to be defined. However, for possibly degenerate domains one can define a relative spectrum by 
	\begin{equation*}
	\spec(\partial U;M):=\{-T-\omega(A):T\in \{0\}\cup\spec(\partial U),\ A\in\pi_2(M)\}.
	\end{equation*}
	and construct spectral killers for Hamiltonians satisfying $\spec(\partial U;M)\cap (0,c(H)]=\emptyset$. The proof in this case is simpler than that of Proposition~\ref{pro:killer1} and does not require Lemma~\ref{lem:radial_Hamiltonians2} below. This observation will be used in the proof of Theorem~\ref{thm:max_ineq_rational} which concerns the max inequality on rational manifolds.
\end{rem}
We start by stating a simple corollary of Lemma~\ref{lem:radial_Hamiltonians1} from the previous section.
\begin{cor}\label{cor:action_radial_Hamiltonian}
	For a radial Hamiltonian $H=\chi(s)$, the action of a capped non-constant 1-periodic orbit $(\gamma,u)$  is given by
	\begin{equation}\label{eq:action_radial_Hamiltonian}
	\cA_H(\gamma,u)=\chi(s) - s\cdot \chi'(s)-\omega(A),
	\end{equation} 
	where $s=s(\gamma)$ is the  Liouville coordinate of the level set containing $\gamma$ and $A\in\pi_2(M)$ is such that $u=u_0\#A$ for a capping disk $u_0\subset \{s=s(\gamma)\}$.
\end{cor}
\begin{proof}
	Recall that Lemma~\ref{lem:radial_Hamiltonians1} states that every non-constant 1-periodic orbit $\gamma$ of a radial Hamiltonian $H=\chi(s)$ corresponds to a Reeb orbit $\hat\gamma\in\cP(\partial U)$ whose action is $\int_{\hat\gamma}\lambda=|\chi'(s)|$. Notice that $\hat \gamma$ and  $\psi^{-\log s}\gamma$ have the same orientation if $\chi'(s)>0$ and opposite orientation if $\chi'(s)<0$. Let $\hat u_0$ be a capping disk of $\hat \gamma$ whose image coincides with $\psi^{-\log s} u_0$. Then,
	\begin{eqnarray}
	\cA_H(\gamma,u) &=& H(\gamma)-\int_u \omega = \chi(s) - \int_{u_0}\omega - \omega(A) 
	\nonumber\\	&=& 
	\chi(s) - \int_{\psi^{-\log s}u_0}\left(\psi^{\log s}\right)^*\omega-\omega(A)
	=\chi(s)-s\cdot \int_{\psi^{-\log s}u_0}\omega -\omega(A) \nonumber\\ 
	&=&\chi(s)-s\cdot \int_{\partial (\psi^{-\log s}u_0)}\lambda -\omega(A) 
	=\chi(s)-s\cdot \sign\chi'(s)\cdot \int_{\hat \gamma}\lambda - \omega(A) 
	\nonumber\\	&=& 
	\chi(s)-s\cdot\chi'(s) -\omega(A),\nonumber
	\end{eqnarray}
	where in the last equality we used equation~(\ref{eq:action_vs_dchi}) from Lemma~\ref{lem:radial_Hamiltonians1}.
\end{proof}

Since radial Hamiltonians are degenerate, we will perturb them into non-degenerate Hamiltonians.
After the perturbation, there may appear several periodic orbits, $\{\gamma_i\}$, in a small neighborhood of every degenerate orbit $\gamma$. In this case, the CZ index of the non-degenerate perturbed orbits is close to the RS index of the original degenerate orbit: 
\begin{equation}\label{eq:RS_after_pert}
|\indCZ_H(\gamma_i,u_i)-\indRS(\gamma,u)|\leq\frac{1}{2}\dim\ker ((d\varphi_H^1)_{\gamma(0)}-\id),
\end{equation} 
see, for example, section 3 of \cite{ishikawa2015spectral}.
The next lemma relates between the RS indices of 1-periodic orbits of radial Hamiltonians and the CZ indices of Reeb orbits. Moreover, it shows that generically, the kernel of $(\varphi_H^1)_{\gamma(0)}-\id$ is 1-dimensional.  
\begin{lemma}\label{lem:radial_Hamiltonians2}
	Let $H=\chi(s)$ be a radial Hamiltonian. For every non-constant 1-periodic orbit $\gamma$ of $H$ and a capping $u_0\subset \{s=s(\gamma)\}$ of $\gamma$, 
	\begin{equation}\label{eq:RS_vs_CZ_R}
		\indRS(\gamma,u_0)=\sign(\chi'(s))\cdot\indCZ_R(\hat \gamma,\hat u_0) +\frac{1}{2}\sign\left(\chi''(s)\right)
	\end{equation}
	where $\hat \gamma(t):=\psi^{-\log s}\gamma(t/\chi'(s))\in\cP(\partial U)$ is the Reeb orbit conjugate to $\gamma$, $\hat u_0$ is a capping disk of $\hat \gamma$ whose image coincides with $\psi^{-\log s}u_0\subset \partial U$ and the sign of zero is considered to be zero.
	Moreover, if $\chi''(s)\neq 0$ for $s=s(\gamma)$, then $\dim\ker ((d\varphi_H^1)_{\gamma(0)}-\id)=1$.
\end{lemma}
\begin{proof}
	In order to relate the two different indices we show that the restrictions of the linearized flows to the contact distribution are conjugated. This can be done by differentiating the conjugation of the Hamiltonian and Reeb flows, given in (\ref{eq:Ham_Reeb_flow_conj}). Extending the contact distribution $\xi$ to a neighborhood of $\partial U$ using the Liouville flow $\psi^{\log s}$, we have
	\begin{eqnarray*}
	d\varphi^t_H|_\xi &=& d\left(\psi^{\log s}\circ \varphi_R^{\chi'(s)\cdot t}\circ\psi^{-\log s}\right)\Big|_{\xi}\\
	&=& d\psi^{\log s} d\varphi_R^{\chi'(s)\cdot t}d\psi^{-\log s}|_\xi,
	\end{eqnarray*}
	where the last equality follows from the contact distribution $\xi$ is tangent to level sets of the Liouville coordinate $s$. We therefore conclude that the linearized Hamiltonian and Reeb flows are conjugated on the contact distribution. By the conjugation property of the RS index for paths of matrices, the RS index of the restriction of $d\varphi_H^t$ to $\xi$ is equal to $\sign(\chi'(s))\cdot\indCZ_R(\hat \gamma,\hat u_0)$. By the product property of the RS index, it remains to compute the index of the restriction of the linearized Hamiltonian flow to $span\{Y,R\}$. We remind that the Reeb vector field is defined wherever the Liouville vector field is, by equations (\ref{eq:Reeb_def}).  
	Since $X_H$ is proportional to $R$, and since $H$ is autonomous, the linearized Hamiltonian flow preserves $R$: $d\varphi^t_H (R)=d\varphi^t_H (\frac{1}{\chi'(s)}X_H)= \frac{1}{\chi'(s)} X_H \circ \varphi^t_H=R\circ \varphi^t_H$. In order to compute the linearized Hamiltonian flow on the Liouville vector field, let us first compute the conjugation of $\varphi_H^t$ under the Liouville flow. Let $x\in U$ and abbreviate $s=s(x)$. Using the conjugation between the Hamiltonian and the Reeb flows (\ref{eq:Ham_Reeb_flow_conj}) from Lemma~\ref{lem:radial_Hamiltonians1}, we have 
	\begin{eqnarray}\label{eq:Ham_flow_conj}
	\psi^{-\tau}\circ\varphi_H^t\circ\psi^\tau(x) &=& \psi^{-\tau}\psi^{\log (se^\tau)} \varphi_R^{\chi'(se^\tau)\cdot t}\psi^{-\log (se^\tau)}\psi^\tau(x)\nonumber\\
	&=& \psi^{\log s} \varphi_R^{\chi'(se^\tau)\cdot t}\psi^{-\log s}(x)\nonumber\\
	&=& \varphi_H^{\frac{\chi'(se^\tau)}{\chi'(s)}\cdot t}(x). 
	\end{eqnarray}
	Using the above, 
	\begin{eqnarray*}
	d\varphi^t_H (Y) &=& \frac{d}{d\tau}\Big|_{\tau=0} \varphi_H^t\circ \psi^\tau
	\overset{(\ref{eq:Ham_flow_conj})}{=} \frac{d}{d\tau}\Big|_{\tau=0}\psi^{\tau}\circ \varphi_H^{\frac{\chi'(se^\tau)}{\chi'(s)}\cdot t}\\
	&=& Y\circ \varphi_H^t +d\psi^\tau X_H\cdot  \frac{se^{\tau}\chi''(se^\tau)}{\chi'(s)}\cdot t\big|_{\tau=0}\\
	&=& Y\circ \varphi_H^t + t\cdot\frac{s\chi''(s)}{\chi'(s)}\cdot X_H =  Y\circ \varphi_H^t + t\cdot{s\chi''(s)}\cdot R.
	\end{eqnarray*}
	
	Fix a trivialization $\hat T:\im(\hat u_0)\times \R^{2n-2}\rightarrow \xi|_{\im(\hat u_0)}$ of the contact distribution $\xi$ along $\im (\hat u_0)\subset \partial U$, then $d\psi^{\log s} \hat T$ is a trivialization of $\xi$ along the image of $u_0$. We can complete it to a trivialization $T:\im(u_0)\times \R^{2n}\rightarrow TM|_{\im(u_0)}$ of $TM$ along the image of $u_0$, using the vector fields $Y$ and $R$, namely $T(x, e_{2n-1})= Y(x)$, $T(x, e_{2n})= R(x)$.
	It follows from the above computations that in this trivialization the matrices $\Phi_H(t)$ representing $d\varphi_H^t$ along $\gamma$ are
	\begin{equation*}
	\Phi_H(t) = \left(\begin{array}{ccc}
	\Phi_R(s\chi'(s)\cdot 	t) &0 &0\\
	0& 1 & 0\\
	0& t\cdot s\chi''(s) &1
	\end{array}\right),
	\end{equation*}
	where $\Phi_R(t)$ are the matrices representing the restriction of $d\varphi_R^t$ to $\xi$ in $T$. By the product property of the RS index, we have
	\begin{equation}\label{eq:CZ_additivity}
	\indRS(\Phi_H) = \sign(\chi'(s))\cdot\indRS(\Phi_R) + \indRS(\Psi),
	\end{equation}
	where 
	\begin{equation*}
	\Psi(t):=\left( \begin{array}{cc}
	1 & 0\\
	t\cdot {s\chi''(s)} &1
	\end{array}\right).
	\end{equation*}
	The path of symplectic matrices $\Psi(t)$ is degenerate. By Proposition~4.9 from \cite{gutt2014generalized}, the RS index of $\Psi$ is\footnote{More accurately, one should apply Proposition~4.9 from \cite{gutt2014generalized} to the path of transposed matrices $\Psi^T$ and use the inverse property of the RS index.}
	 $$
	 \indRS (\Psi) =\sign\left({s\chi''(s)}\right)/2=\sign (\chi''(s))/2.
	 $$
	 Together with equation (\ref{eq:CZ_additivity}), this implies that $\indRS(\gamma,u_0)=\sign(\chi'(s))\cdot \indCZ_R(\hat \gamma,\hat u_0)+\frac{1}{2}\sign(\chi''(s))$. This proves formula (\ref{eq:RS_vs_CZ_R}). It remains to prove that, if $\chi''(s)\neq 0$, the  kernel of $(d\varphi_H^1)_{\gamma(0)}-\id$ is one-dimensional. Indeed, since $\hat \gamma$ is non-degenerate, the kernel of $\Phi_H(1)-\id$ coincides with that of $\Psi(1)-\id$, which is one-dimensional when $\chi''(s)\neq 0$.
\end{proof}

Having established Lemma~\ref{lem:radial_Hamiltonians2}, we are now ready to prove Proposition~\ref{pro:killer1}.
\begin{proof}[Proof of Proposition~\ref{pro:killer1}]
	Let $H$ be a Hamiltonian supported in $U$. In what follows, we construct a continuous deformation $H_\tau$ of the Hamiltonian $H$ and, after perturbing into non-degenerate Hamiltonians, follow the corresponding bifurcation diagram of the $n$-spectrum  (that is, the actions of capped orbits of index $-n$). We remind that by the spectrality property for non-degenerate Hamiltonians, the spectral invariant lies in the $n$-spectrum.
	
	Let $\cN(\partial U)\cong\partial U\times(1-3\epsilon, 1+\epsilon)$ be a small enough neighborhood  of the boundary on which the Liouville coordinate is defined and such that $H|_{\cN(\partial U)}=0$.
	\begin{figure}
		\centering
		\includegraphics[scale=0.7]{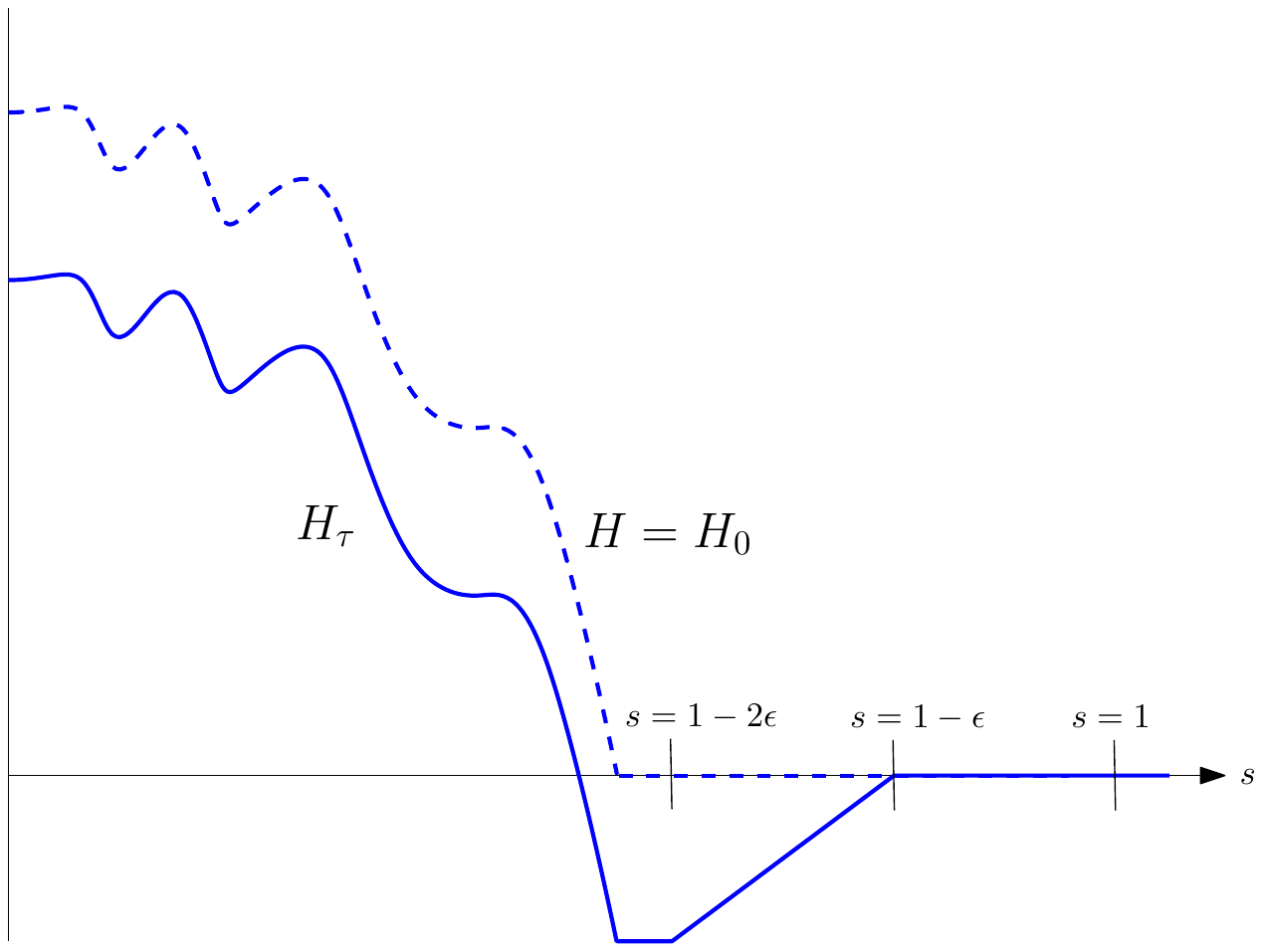}
		\caption{\small{An illustration of the graphs of $H_\tau$ (solid line) and $H$ (dashed line) in the radial coordinate $s$. 
		}}
		\label{fig:deformation}
	\end{figure}
	\noindent
	Consider the autonomous radial Hamiltonians defined by
	\begin{equation}
	K_1(x) :=\begin{cases}
	-1 & x\in U\setminus\big(\partial U\times (1-2\epsilon,1)\big),\\
	\chi_1(s(x)) & x\in\partial U\times(1-2\epsilon,1-\epsilon),\\
	0 & \text{elsewhere}.
	\end{cases}
	\end{equation}
	where $\chi_1:\R\rightarrow\R$ is a smooth approximation of the continuous piecewise linear function taking the value $-1$ for $s\leq 1-2\epsilon$ and $0$ for $s\geq1-\epsilon$, and which coincides with the piecewise linear function outside of a neighborhood of the corners at $1-2\epsilon$ and $1-\epsilon$. We also require that $\chi_1''(s)$ will be strictly negative near the corners $1-2\epsilon$ and $1-\epsilon$. Consider the family of radial Hamiltonians given by $\{K_\tau:=\tau\cdot K_1\}_{\tau\in[0,c(H)]}$ and set $\chi_\tau:=\tau\cdot \chi_1$.
	Note that since the Reeb spectrum is a discrete set, for a generic $\tau\in [0,c(H)]$ non-constant 1-periodic orbits of $K_\tau$ appear only near $s=1-2\epsilon$ and $s=1-\epsilon$.
	We define the deformation of the Hamiltonian $H$ to be $\{H_\tau := H+K_\tau\}_{[0,c(H)]}$, see Figure~\ref{fig:deformation}. Let us follow the change of the $n$-spectrum of a non-degenerate perturbation of $\{H_\tau\}_{\tau\geq 0}$. We will show that the corresponding bifurcation diagram is (approximately) composed of horizontal lines, corresponding to orbits appearing for $s> 1-3\epsilon/2$, and of lines with slope $-1$, corresponding to capped orbit with $s<1-3\epsilon/2$, see Figure~\ref{fig:bifurcation_diag}. 
	Let us split into four regions:
	\begin{itemize}
		\item In $M\setminus {U}$: Here the $H_\tau$ is zero for all $\tau$. After perturbing into non-degenerate Hamiltonians, the actions of capped orbits with CZ index $-n$ in this region are contained in a small neighborhood of $\spec_n(0)\subset \spec_n(\partial U;M)$ and do not change with $\tau$.
		
		\item In $U\setminus\cN(\partial U)$: Here $H_\tau=H-\tau$, and therefore the periodic orbits remain the same as $\tau$ changes, and the action of each orbit decreases linearly: 
		$$
		\cA_{H_\tau}(\gamma,u) = \int_0^1 H_\tau\circ\gamma\ dt- \int_u \omega = \int_0^1 H\circ\gamma\ dt-\tau- \int_u \omega=\cA_H(\gamma,u)-\tau.
		$$
		Therefore, the action spectrum of $H_\tau$ in this region changes linearly in $\tau$, with slope $-1$.
		After perturbing $H$ into a non-degenerate Hamiltonian in this region $H_\tau= H-\tau$ are non-degenerate as well, and the action spectrum consists of lines with slope $-1$.
		
		\item Near $s=1-2\epsilon$: Here  $H_\tau$ coincides with the radial Hamiltonian $K_\tau=\chi_\tau(s)$. Lemma~\ref{lem:radial_Hamiltonians1} states that every non-constant 1-periodic orbit $\gamma$ of $K_\tau$  corresponds to a Reeb orbit of action $\chi_\tau'(s(\gamma))$. Since $\chi_\tau'(s)$ takes values between $0$ and $\tau/\epsilon$ in this region, $H_\tau$ may admit more 1-periodic orbits as $\tau$ grows. The action of each orbit, once it appeared, decreases approximately  linearly in $\tau$, since the value of $H_\tau$ does. Indeed, by Corollary~\ref{cor:action_radial_Hamiltonian} the action with respect to a radial Hamiltonian is given by
		$$
		\cA_{H_\tau}(\gamma,u)=\chi_\tau(s)-\chi_\tau'(s)\cdot s -\omega(A) \approx -\tau- (1-2\epsilon)\cdot \int_{\hat\gamma} \lambda-\omega(A), 
		$$	
		where $s=s(\gamma)$, $\hat\gamma(t)=\psi^{-\log s}\gamma(t/\chi'(s))\in\cP(\partial U)$ and $A:=u\#\bar u_0\in\pi_2(M)$. 
		After perturbing $H_\tau$ into non-degenerate Hamiltonians, the actions attained by orbits in this region remain in a neighborhood of lines with slope $-1$ with respect to $\tau$.
		
		\begin{figure}
			\centering
			\includegraphics[scale=0.5]{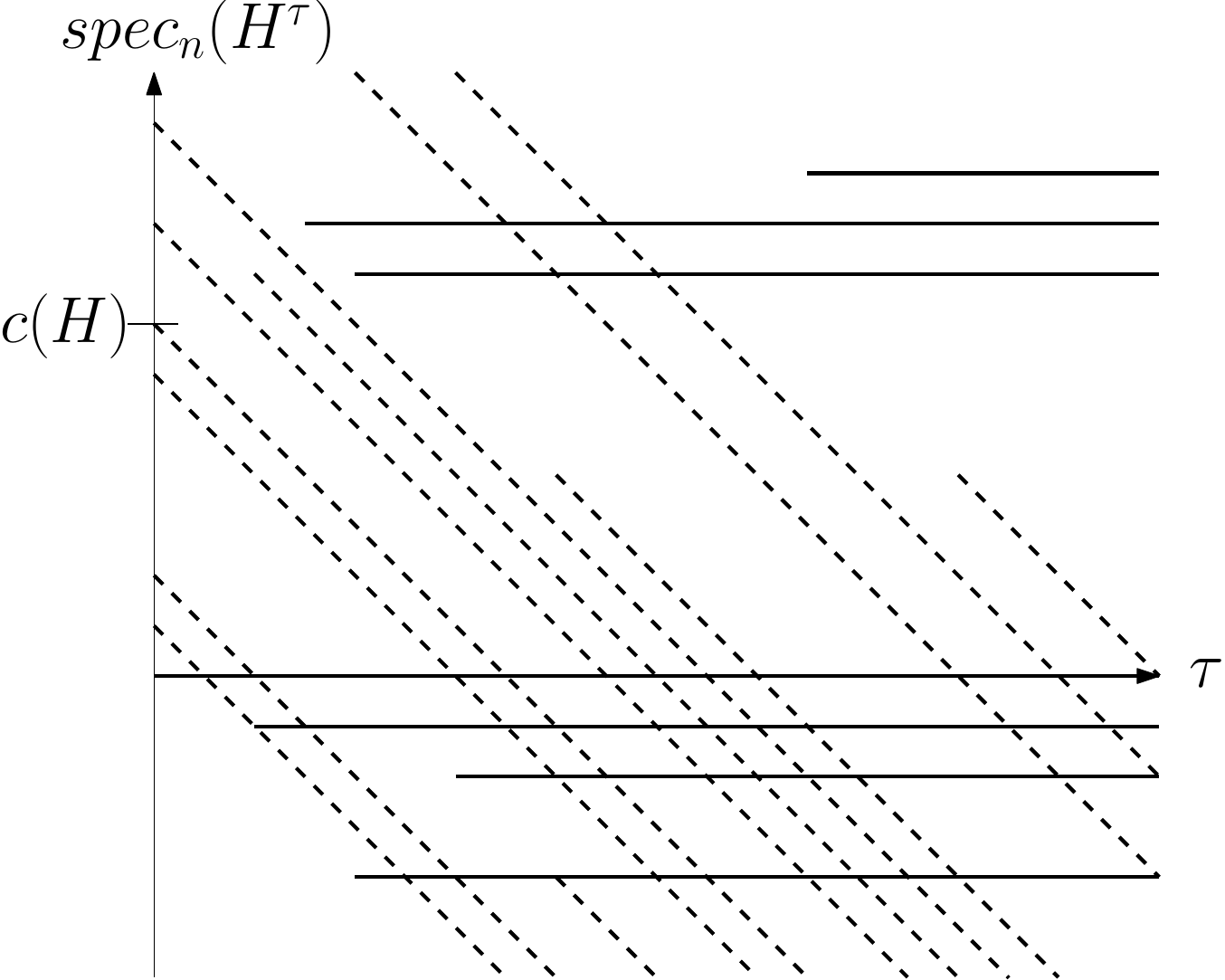}
			\caption{\small{An illustration of the $n$-spectrum of a non-degenerate perturbation of $H_\tau$. The dashed lines correspond to actions of orbits in $\{s<1-3\epsilon/2,1\}$ while the solid lines correspond to actions of  orbits in $\{s>1-3\epsilon/2\}$.}}
			\label{fig:bifurcation_diag}
		\end{figure}	\noindent	
		\item Near $s=1-\epsilon$: Here again $H_\tau$ coincides with the radial Hamiltonian $K_\tau =\chi_\tau(s)$ and by Lemma~\ref{lem:radial_Hamiltonians1}, non-constant 1-periodic orbits correspond to Reeb orbits of action $\chi_\tau'(s)$, which takes values between $0$ and $\tau/\epsilon$. As before, $H_\tau$ may admit more 1-periodic orbits as $\tau$ grows, but this time the action of each orbit, once it appeared, remains approximately constant when $\tau$ varies: 
		$$
		\cA_{H_\tau}(\gamma,u)=\chi_\tau(s)-\chi_\tau'(s)\cdot s -\omega(A)\approx 0-(1-\epsilon)\cdot\int_{\hat\gamma} \lambda -\omega(A),
		$$ 
		where again $s=s(\gamma)$, $\hat \gamma(t)=\psi^{-\log s}\gamma(t/\chi'(s))\in \cP(\partial U)$ and $A:=u\#\bar u_0\in\pi_2(M)$. Let us now show that, after perturbing $H_\tau$ into non-degenerate Hamiltonians, the actions of capped orbits with index $-n$ lie in a small neighborhood of the relative $n$-spectrum $\spec_n(\partial U;M)$. For this end, we need to compare the RS index of 1-periodic orbits of $H_\tau$ with the contact CZ indices of the corresponding Reeb orbits.    
		By our choice of $\chi_\tau$, its first derivative is strictly decreasing in this region, namely, $\chi_\tau''<0$, we can use Lemma~\ref{lem:radial_Hamiltonians2} to conclude that the  RS index of $(\gamma,u_0)$ and the contact CZ index of $(\hat \gamma,\hat u_0)$ are related by $\indRS(\gamma,u_0) =\indCZ_R(\hat \gamma,\hat u_0)-\frac{1}{2}$. 
		After perturbing $H_\tau$ into a non-degenerate Hamiltonian, every non-degenerate capped orbit $(\gamma',u_0')$ must have CZ index close to the RS index of $(\gamma,u_0)$. More accurately, since $\chi_\tau''\neq 0$, Lemma~\ref{lem:radial_Hamiltonians2} guarantees that the kernel of $d\varphi_H^t(\gamma(0))-\id$ is one-dimensional and, by (\ref{eq:RS_after_pert}),
		$|\indCZ_H(\gamma',u_0')-\indRS(\gamma,u_0)|\leq \frac{1}{2}$. We therefore conclude that capped orbits of index $-n$ could appear only for a capping $u'=u_0'\#A$ for $A\in\pi_2(M)$ such that 
		\begin{eqnarray}
		\frac{1}{2}&\geq& |\indCZ_H(\gamma',u_0')-\indRS(\gamma,u_0)| = |-n-2c_1(A)-\indRS(\gamma,u_0)| \nonumber\\
		&=& \Big|-n-2c_1(A) -\indCZ_R(\hat \gamma,\hat u_0)+\frac{1}{2} \Big|. \nonumber
		\end{eqnarray}
		Since $c_1$ and $\indCZ_R$ take integer values, this is equivalent to 
		$$
		c_1(A)=\Big\lceil \frac{-\indCZ_R(\hat \gamma,\hat u_0)-n}{2}\Big\rceil,
		$$
		which implies that the action of the index $-n$ orbit $(\gamma',u')$ is approximately $-\int_{\hat\gamma} \lambda -\omega(A) \in\spec_n(\partial U;M)$.
	\end{itemize}
	
	This analysis shows that the bifurcation diagram corresponding to the $n$-spectrum of a non-degenerate perturbation of $\{H_\tau\}_\tau$ is composed of (neighborhoods of) decreasing lines of slope $-1$, corresponding to orbits in $\{s<1-3\epsilon/2\}$, and of (neighborhoods of) horizontal lines with values in $\spec_n(\partial U;M)$, corresponding to orbits in $\{s>1-3\epsilon/2\}$, as illustrated in Figure~\ref{fig:bifurcation_diag}. 

	The condition that $\spec_n(\partial U;M)$ does not intersect $(0,c(H)]$, namely, (\ref{eq:the_condition}), guarantees that there are no horizontal lines in the action window $(0,c(H)]$. As a consequence, the bifurcation diagram contains no intersections in this window. By the continuity of $c(H_\tau)$ (with respect to the parameter $\tau$) the spectral invariant $c(H_\tau)$ must move along a single line, which implies that $c(H_\tau)=c(H)-\tau$ for $\tau\leq \tau_0:=c(H)$. In particular, $c(H_{\tau_0}) =0$. 	
	This proves the proposition for $K:=K_{\tau_0}$, as $c(H+K)=c(H_{\tau_0}) =0$ and $\|K\|_{C^0}= \tau_0= c(H)$. 
\end{proof}

\section{Deducing the max inequality for various settings.}\label{sec:proving_max_ineq}
Proposition~\ref{pro:killer1}, together with Claim~\ref{clm:max_ineq}, imply that the max inequality (\ref{eq:max_ineq}) holds for any collection of  Hamiltonians $H_1,\dots,H_N$, supported in disjoint domains $U_1,\dots,U_N$ with incompressible contact type boundaries, respectively, provided that condition (\ref{eq:the_condition}) is satisfied for each pair $(H_i, U_i)$. Namely, that for each $i$, the {relative} $n$-spectrum $\spec_n(\partial U_i;M)$ does not intersect the interval $(0,c(H_i)]$.
In this section we prove Theorems \ref{thm:max_ineq_rational}, \ref{thm:max_ineq_neg_monotone} and \ref{thm:max_ineq_pos_monotone}, by showing that condition (\ref{eq:the_condition}) is satisfied under the assumptions stated in the theorems.
Before that, let us show that condition (\ref{eq:the_condition}) always holds on symplectically aspherical manifolds.
 
\subsection{Symplectically aspherical manifolds.} 
When $(M,\omega)$ is symplectically aspherical, namely, $\omega|_{\pi_2(M)}=0$, the max inequality (\ref{eq:max_ineq}) holds for every Hamiltonians  $H_1,\dots,H_N$, supported in disjoint domains $U_1,\dots,U_N$ with incompressible contact type boundaries. Indeed, given a domain $U$ with  incompressible contact type boundary, the fact that $\omega$ vanishes on $\pi_2(M)$ implies that
\begin{equation*}
\spec_n(\partial U;M)\subset -\spec(\partial U)\cup\{0\} \subset(-\infty,0]
\end{equation*}
and in particular $\spec_n(\partial U;M)$ does not intersect the interval $(0,c(H)]$ for any Hamiltonian $H$ with $c(H)>0$. Therefore condition (\ref{eq:the_condition}) is satisfied with no additional assumptions on $U$.

We remind that the max inequality is a weaker statement than the max formula proved by Humili\`ere, Le Roux and Seyfaddini in \cite{humiliere2016towards}, which states that the spectral invariant of the sum of Hamiltonians supported in disjoint incompressible Liouville domains on symplectically aspherical manifolds is equal to the maximal spectral invariant of the summands.

\subsection{Rational manifolds.}
Let $(M,\omega)$ be a rational symplectic manifold, namely  $\omega(\pi_2(M))=\kappa\Z$ for some $\kappa\in\R$. In this section we prove that the max inequality (\ref{eq:max_ineq}) holds for Hamiltonians $H_i$ supported in disjoint domains $U_i$ with incompressible contact type boundaries, if $\spec(\partial U_i)\subset T_i\Z$ for some $T_i|\kappa$ and  $c(H_i)<T_i$, respectively.

\begin{proof}[Proof of Theorem~\ref{thm:max_ineq_rational}]
	Let $U\subset M$ be a domain with an incompressible contact type boundary and suppose there exists $T|\kappa$ such that $\spec(\partial U)\subset T\Z$. In order to be able to apply Proposition~\ref{pro:killer1} and Remark~\ref{rem:killer_for_degenerate}, we need to show that the relative spectrum $\spec(\partial U;M)$ does not intersect the interval $(0,c(H)]$ for every Hamiltonian $H$ supported in $U$ with $c(H)<T$. The max inequality for such Hamiltonians will then follow from Claim~\ref{clm:max_ineq}. 
	When $\spec(\partial U)\subset T\Z$, the relative spectrum $\spec(\partial U;M)$ is contained in the set $-\{0\}\cup\spec(\partial U)+\kappa \Z\subset T\Z+\kappa \Z$.  Since $T|\kappa$, the relative spectrum is contained in the lattice $T\Z$, and the intersection $\spec(\partial U;M)\cap (0,c(H)]$ is empty for every Hamiltonian with $c(H)<T$.
\end{proof}

\subsection{Negatively monotone manifolds.}
Let $(M,\omega)$ be a negatively monotone symplectic manifold, namely, $\omega=\kappa\cdot c_1$ on ${\pi_2(M)}$ for some $\kappa\leq 0$. We now prove Theorem~\ref{thm:max_ineq_neg_monotone}, which states that the max inequality (\ref{eq:max_ineq}) holds for Hamiltonians supported in disjoint domains with incompressible contact type boundaries, such that the contact CZ indices of the Reeb orbits are all non-negative.  

\begin{proof}[Proof of Theorem~\ref{thm:max_ineq_neg_monotone}]
Let $U\subset M$ be a domain with an incompressible contact type boundary, such that the contact CZ index of every Reeb orbit on $\partial U$ is non-negative, namely
\begin{equation*}
\indCZ_R(\gamma,u_0) \geq 0 \text{ for all }\gamma\in\cP(\partial U) \text{ and } u_0\subset \partial U.
\end{equation*}
In what follows we show that the relative $n$-spectrum $\spec_n(\partial U;M)$ is non-positive and in particular does not intersect the interval $(0,c(H)]$ for every Hamiltonian $H$ supported in $U$. This will establish condition (\ref{eq:the_condition}) and will enable us to  conclude the max inequality from  Proposition~\ref{pro:killer1} and Claim~\ref{clm:max_ineq}. 
The relative $n$-spectrum $\spec_n(\partial U;M)$ contains terms coming from the action spectrum of the zero function, $\spec_n(0)$, and terms coming from actions of Reeb orbits. Starting with $\spec_n(0)$, it is composed of $-\omega(A)$ for $A\in \pi_2(M)$ such that $c_1(A)\in\{-n,\dots,0\}$. In this case, $-\omega(A)=-\kappa c_1(A)\leq 0$, and the spectrum is indeed non-positive. The rest of the elements in the relative $n$-spectrum are of the form $-\int_{ \gamma}\lambda -\omega(A)$, where $A$ is such that  $c_1(A)=\big\lceil \frac{k-n}{2}\big\rceil$ and  $k=-\indCZ_R(\gamma,u_0)$ for some Reeb orbit $\gamma\in\cP(\partial U)$ with respect to a capping disk $u_0\subset\partial U$. Since we assumed that the contact CZ index of every capped Reeb orbit is non-negative, $k$ is non-positive. Therefore, the Chern class of $A$ is non-positive as well, i.e., $c_1(A)\leq 0$. We conclude that $-\int_\gamma \lambda-\omega(A) = -\int_\gamma \lambda-\kappa\cdot c_1(A)\leq 0$ as required. 
\end{proof}

\subsection{Positively monotone manifolds.}
Let $(M,\omega)$ be a positively monotone symplectic manifold, namely, $\omega=\kappa\cdot c_1$ on ${\pi_2(M)}$ for some $\kappa> 0$, and assume in addition that its dimension is greater than 2. We now prove Theorem~\ref{thm:max_ineq_pos_monotone}, which states that the max inequality (\ref{eq:max_ineq}) holds for Hamiltonians $H_i$ supported in disjoint domains $U_i$ with incompressible dynamically convex boundaries, such that for each $i$, $C(U_i)\leq \kappa$ and $c(H_i)<\kappa$.

\begin{proof}[Proof of Theorem~\ref{thm:max_ineq_pos_monotone}]
	Let $U\subset M$ be a domain with an incompressible dynamically convex boundary, such that $C(U)\leq \kappa$, where $C(U)$ is the invariant from Definition~\ref{def:ratio_invariant}. In what follows we show that $\spec_n(\partial U;M)$ does not intersect the interval $(0,\kappa)$. As a consequence, condition (\ref{eq:the_condition})  will follow for every Hamiltonian $H$ supported in $U$, such that $c(H)<\kappa$. 
	Starting with $\spec_n(0)$, which is the first component of $\spec_n(\partial U;M)$, we see that $-\omega(A) = -\kappa c_1(A)\in\{0,\kappa,\dots, n\kappa\}$, and therefore  $\spec_n(0)\cap(0,\kappa)=\emptyset$.
	The rest of the elements of the relative spectrum $\spec_n(\partial U;M)$ are of the form $-\int_\gamma\lambda  - \omega(A)$, where $\gamma$ is a periodic Reeb orbit and $A\in\pi_2(M)$ is such that 
	$$
	c_1(A)=\Big\lceil\frac{-\indCZ_R(\gamma, u_0)-n}{2}\Big\rceil
	,
	$$
	for a capping disk $ u_0\subset\partial U$ of $\gamma$.
	Recalling that the invariant $C(U)$ was defined to be the supremum over ratios of the form $2\int_{\gamma}\lambda/(\indCZ_R(\gamma,u_0) - n+1)$, we can use it to bound the absolute value of $c_1(A)$ from below:
	\begin{equation*}
	-c_1(A) \geq \frac{\indCZ_R(\gamma,u_0)+n-1}{2}= \frac{\indCZ_R( \gamma, u_0)-n+1}{2} + n-1    
	\geq \frac{\int_\gamma\lambda}{C(U)}+n-1,
	\end{equation*}
	where, in the last inequality, we used our assumption that $\partial U$ is dynamically convex, and hence $\indCZ_R(\gamma, u_0)-n+1> 0$.
	In light of this observation, the action can be bounded as follows:
	\begin{eqnarray}
	-\int_{ \gamma}\lambda -\omega(A) &=& 
	-\int_{ \gamma}\lambda -c_1(A)\cdot\kappa \geq  -\int_{ \gamma}\lambda +{\kappa}\cdot\left(\frac{\int_{ \gamma}\lambda}{C( U)}+n-1\right)\nonumber\\
	&=& -\int_{ \gamma}\lambda +\frac{\kappa}{C( U)}\cdot\int_{ \gamma}\lambda +(n-1)\kappa	
	\geq (n-1)\kappa, \nonumber   
	\end{eqnarray}
	where the last inequality follows from our assumption that $C( U)\leq \kappa$. Since the dimension of $M$ is $2n$ and is assumed to greater than $2$,  we conclude that the action $-\int_{ \gamma}\lambda -\omega(A)$ is bounded from below by $\kappa$. 
	As a result, the relative spectrum, $\spec_n(\partial U;M)$, does not intersect the interval $(0,\kappa)$, which, by our assumption, contains $(0,c(H)]$. Having established condition (\ref{eq:the_condition}),  Proposition~\ref{pro:killer1} together with Claim~\ref{clm:max_ineq} guarantee that the max inequality holds in this setting.  
\end{proof}

We end this section with a proof of Claim~\ref{clm:sphere}, which states that the max inequality holds for Hamiltonians supported in certain disjoint disks on the sphere.
\begin{proof}[Proof of Claim~\ref{clm:sphere}]
We assume that the area form $\omega$ is normalized so that the total area of the sphere is 1. In this case, the monotonicity constant is $\kappa=\frac{1}{2}$. 
Let $U\subset S^2$ be a disk of area $a$ and let $H$ be a Hamiltonian supported in $U$. In what follows we show that if $a\notin(1/3,1/2)$, one can construct a spectral killer for $H$. Claim~\ref{clm:max_ineq} will then guarantee that the max inequality (\ref{eq:max_ineq}) holds for such Hamiltonians.
Consider the family of radial Hamiltonians $K_\tau:=\tau\cdot K_1$ that was constructed in the proof of Proposition~\ref{pro:killer1}. Recall that $K_1$ is supported in $U$, constant and equal to $-1$ for $s<1-2\epsilon$, and is approximately linearly increasing for $s\in(1-2\epsilon, 1-\epsilon)$. Here and in what follows $s$ is the Liouville coordinate on the disk $U$. In order to conclude that $K_{c(H)}$ is a spectral killer for $H$, we need to show that the actions of index $-n$ capped periodic orbits of $K_\tau$ for $s\geq 1-3\epsilon/2$ do not intersect the interval $(0,c(H)]$. After perturbing $K_\tau$ into a non-degenerate Hamiltonian, its 1-periodic orbits for $s\geq 1-3\epsilon/2$ are $\{\gamma^k_1,\gamma^k_2\}_k$ and $p$, where $\gamma^k_i$ rotates $k$ times around $\partial U$ and $p\in U^c$ is a maximum point of action approximately zero. When paired with capping disks $u_{i,0}^k\subset U$, the actions of $\gamma^k_i$ are approximately $-k\cdot area(U) = -k\cdot a$, and their CZ indices are
\begin{equation*}
\indCZ_H(\gamma^k_1,u_{1,0}^k) =2k,\quad \indCZ_H(\gamma^k_2,u_{2,0}^k) =2k-1.
\end{equation*}
See, for example, \cite[p.11]{seyfaddini2014spectral}\footnote{Note that the sign choice for the CZ index in \cite{seyfaddini2014spectral} is opposite to ours.}. Therefore, there exists a capping disk $u_i^k$ such that $\indCZ_H(\gamma_i^k,u_i^k)=-n=-1$ if and only if $i=2$ and $u_2^k=u_{2,0}^k\#A^k$, for $A^k\in\pi_2(S^2)$ such that $c_1(A^k)=-k$. In this case, the action of the index $-1$ capped orbit is approximately
$$
-k\cdot a-\omega(A^k)=-k\cdot a-\kappa\cdot c_1(A^k)= -k\cdot a-{1}/{2}\cdot c_1(A^k)=k \cdot (1/2-a).
$$
Overall, the actions of non-degenerate orbits of index $-1$  are $\{0, k\cdot (1/2-a)\}$, and we can construct a spectral killer for $H$ if these actions do not intersect the interval $(0,c(H)]$. Let us show that this holds whenever $a\notin(1/3,1/2)$. Clearly, if the area $a$ of $U$ is greater than or equal to $1/2$, then $k(1/2-a)$ is non-positive and does not lie in the interval $(0,c(H)]$. Any open disk in $S^2$ of area less than $1/2$ is known to be displaceable, and its displacement energy coincides with the area. The Hamiltonian $H$ is compactly supported in $U$, and therefore is also supported in a slightly smaller disk. Applying the energy-capacity inequality to the smaller disk we conclude that $c(H)<a$. Therefore, $(0,c(H)]\subset(0,1/3)$ when $a\leq 1/3$. On the other hand, recalling that the minimal Chern number on $S^2$ is 2 and that $k=-c_1(A^k)$, we conclude that $k(1/2-a)\geq 2(1/2-a)\geq 2\cdot (1/2-1/3)=1/3$ and in particular $k(1/2-a)$ does not lie in $(0,c(H)]$.
\end{proof}

\section{Uniform bounds on spectral invariants.}\label{sec:unif_bounds}
In this section we prove Theorems~\ref{thm:spec_bound_pos} and \ref{thm:spec_bound_neg}, which state uniform bounds for the spectral invariants of Hamiltonians supported in portable Liouville domains with dynamically convex incompressible boundaries, in monotone manifolds. When the manifold is positively monotone, one has to add a condition regarding the ``size" of the domain containing the support, as seen in the following example.
\begin{exam}\label{exa:equator_sh}
	Let $(M,\omega)$ be the two-dimensional sphere, endowed with an area form. It is known that the equator $E\subset S^2$ is superheavy, see, e.g., \cite{polterovich2014function}. In \cite{entov2004quasi}, Entov and Polterovich proved that 
	the spectral invariant of any Hamiltonian is not smaller than the minimal value the Hamiltonian attains on a superheavy set. Therefore, when $U$ is a large disk containing the equator, there is no uniform upper bound for the spectral invariant of Hamiltonians supported in $U$.   
\end{exam}

In order to bound the spectral invariant of a general Hamiltonian supported in a Liouville domain $U$, it is enough to consider simple, arbitrarily large radial Hamiltonians and use the monotonicity property of spectral invariants. 
Note that, on Liouville domains, we may consider radial Hamiltonian (in the sense of Definition~\ref{def:radial_hamiltonian}) that are not necessarily locally constant outside of a neighborhood of the boundary, but are locally constant on a small neighborhood of the core.

\subsection{Positively monotone manifolds.} \label{subsec:bound_pos_mon}
Let $(M,\omega)$ be a positively monotone symplectic manifold, namely, $\omega=\kappa\cdot c_1$ on ${\pi_2(M)}$ for $\kappa> 0$. Let us prove Theorem~\ref{thm:spec_bound_pos}, which states that for every Hamiltonian supported in a disjoint union $U$ of portable Liouville domains with incompressible dynamically convex boundaries such that $C(U)\leq \kappa$, the spectral invariant is smaller than $C(U)$. In what follows we concern the Liouville coordinate of $U$, by which we mean the coordinate on each connected component.

\begin{proof}[Proof of Theorem~\ref{thm:spec_bound_pos}]
	As mentioned above, in order to prove upper bounds for spectral invariants of Hamiltonians supported in $U$, it is enough to consider (non-degenerate perturbations of) non-increasing radial Hamiltonians. The claim for general Hamiltonians will then follow from the monotonicity property of spectral invariants. Let $H:=\chi(s)$ where $s$ is the Liouville coordinate on $U$ and $\chi:\R\rightarrow\R$ is a smooth approximation of a piecewise linear function that is constant for $s\leq \varepsilon$, vanishes for $s\geq 1-\delta$ (here $\delta>0$ is arbitrarily small) and is linearly decreasing in between. We choose $\chi$ such that, outside of a neighborhood of the ``corners" $s=\varepsilon$ and $s=1-\delta$,  its derivative is constant and does not lie in the Reeb spectrum of $\partial U$. We also assume that $\chi''$ is strictly positive outside of the intervals in which $\chi'$ is constant (and, in particular, near the ``corners"). Finally, we choose $\varepsilon$ to arbitrarily smaller than $\frac{\max\chi}{1-\delta}$, so that $\varepsilon\cdot \chi'$ is arbitrarily small. 
	Let us prove that $c(H)<C(U)$ in two steps. First we  use a continuous deformation of $H$ to show that $c(H)$ is bounded by the maximal action that can be possibly attained by index $-n$ capped orbits near $s=1-\delta$ for Hamiltonians of this shape, and then we show that these actions are all smaller than $C(U)$.\\

	\noindent\underline{Step 1:} 
	In this step we consider a continuous deformation of the Hamiltonian $H$, and describe the bifurcation diagram of its spectrum.
	\begin{figure}
		\centering
		\includegraphics[scale=0.7]{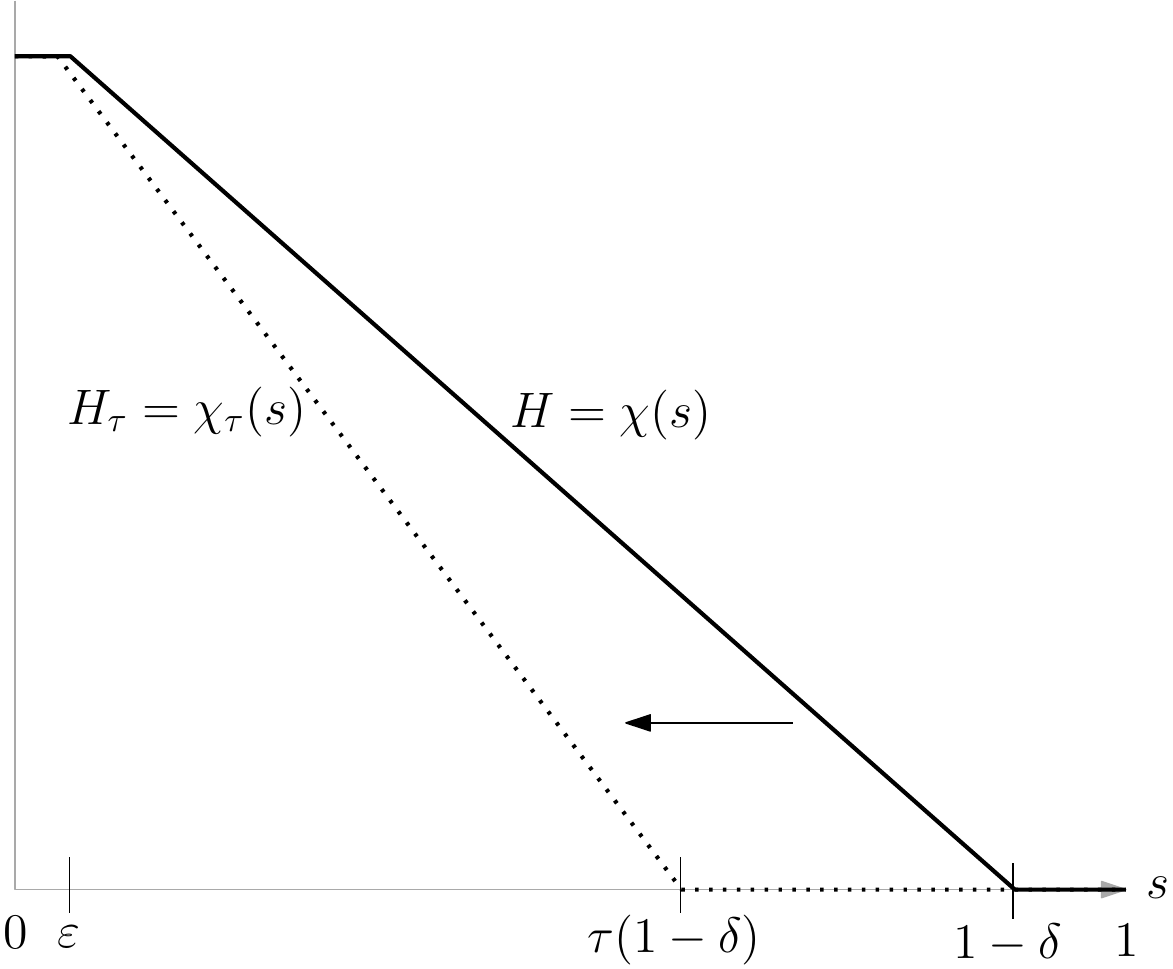}
		\caption{\small{The deformation of the radial Hamiltonian $H$ is given by shrinking its support, using the Liouville flow. Generically, non-constant 1-periodic orbits appear near the ``corners", namely near $s=\tau\varepsilon$ and $s=\tau(1-\delta)$.}}
		\label{fig:deformation_pos_bound_c}
	\end{figure}
	The deformation is given by composing $H$ on the inverse  Liouville flow:
	\begin{equation}\label{eq:deformation_bound_pos}
	H_\tau:= \begin{cases}
	H\circ \psi^{-\log \tau} & \text{ on } \psi^{\log \tau}U,\\
	0 & \text{ on } (\psi^{\log \tau}U)^c
	\end{cases}
	\end{equation}
	for $\tau\in(0,1]$, see Figure~\ref{fig:deformation_pos_bound_c} for an illustration. Notice that $H_\tau =\chi_\tau(s):= \chi(s/\tau)$ is a radial Hamiltonian for each $\tau$, 	
	its non-constant 1-periodic orbits are in correspondence with the Reeb orbits on $\partial U$ and we can use Lemmas \ref{lem:radial_Hamiltonians1}, \ref{lem:radial_Hamiltonians2} to relate their actions and indices. 
	Abbreviating $s=s(\gamma)$, Lemma~\ref{lem:radial_Hamiltonians1} states that  $\hat{\gamma}(t):=\psi^{-\log s}\gamma(t/\chi'(s))\subset\partial U$ is a Reeb orbit of action $-\chi_\tau'(s)\in \spec(\partial U)$.
	Recalling that the Reeb spectrum is a discrete set, we conclude that for generic $\tau\in(0,1]$, non-constant 1-periodic orbits of $H_\tau$ appear only near the ``corners" $s=\tau\varepsilon$ and $s=\tau(1-\delta)$.
	As stated in Corollary~\ref{cor:action_radial_Hamiltonian}, the action of a capped 1-periodic orbit $(\gamma,u)$ of $H_\tau$ is given by 
	\begin{equation*}
	\cA_{H_\tau}(\gamma,u)=\chi_\tau(s)-s\cdot\chi_\tau'(s)-\omega(A),
	\end{equation*}
	where $A\in \pi_2(M)$ is such that $u=u_0\#A$ and $u_0\subset\{s=s(\gamma)\}$.
	Let us describe the bifurcation diagram of the spectrum of $H_\tau$ as $\tau$ varies. When $\tau$ decreases, more 1-periodic orbits may appear near the corners $s=\tau\varepsilon$ and $s=\tau(1-\delta)$. The action of a capped orbit with $s=s(\gamma)$ near $\tau\varepsilon$, once it appeared, is approximately 
	$$ 
	\chi_\tau(\tau\varepsilon) -\varepsilon\tau\chi_\tau'(s)  -\omega(A) =	\chi(\varepsilon) -\varepsilon\tau\frac{1}{\tau}\chi'(s/\tau)  -\omega(A)  \approx \chi(0)-\omega(A),
	$$
	where the last approximation is due to our assumption that $\varepsilon$ is arbitrary small compared to the derivative of $\chi$. In particular, the action of orbits near  $s=\tau\varepsilon$ remains approximately constant as $\tau$ varies.  This is also the case for capped constant 1-periodic orbits (namely, pairs of a critical point of $H_\tau$ and a sphere $A\in\pi_2(M)$) in the region $s\geq \tau(1-\delta)$, in which case the action is approximately $-\omega(A)$. On the other hand, the action of a non-constant 1-periodic orbit near $s=\tau(1-\delta)$, once it appeared, is approximately $0-\tau(1-\delta)\cdot\chi_\tau'(s)-\omega(A)$ and, in particular, changes linearly in $\tau$ with slope $-(1-\delta)\chi_\tau'(s)\in(1-\delta)\cdot\spec(\partial U)$. We conclude that the bifurcation diagram is (approximately) composed of horizontal lines, corresponding to orbits near $s=\tau\varepsilon$, as well as constant orbits, and of lines with positive slopes, corresponding to orbits near $s=\tau(1-\delta)$, see Figure~\ref{fig:bifurcation_diag_pos_bound_c}.
	\begin{figure}
		\centering
		\includegraphics[scale=0.7]{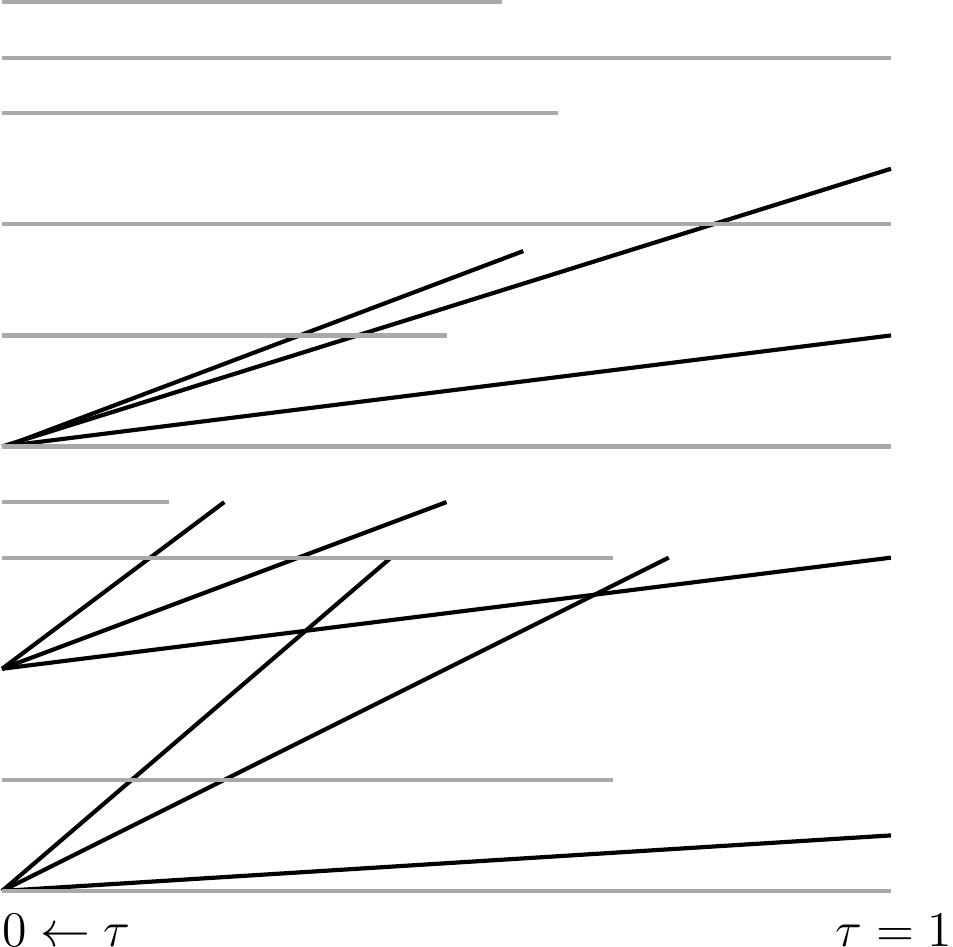}
		\caption{\small{The bifurcation diagram corresponding to the $n$-spectrum of $H_\tau$. Since the spectral invariant starts close to zero and moves continuously on the diagram, its value is bounded by the maximal height attained by a non-horizontal line.}}
		\label{fig:bifurcation_diag_pos_bound_c}
	\end{figure}
	When $\tau$ is very small, $H_\tau$ is supported in $\psi^{-\log\tau} U$ which is displaceable with small displacement energy, as explained in Section~\ref{subsec:portable_preliminaries} above. By the energy-capacity inequality, $c(H_\tau)$ is bounded by the displacement energy of the support and hence is very small. Following the bifurcation diagram of the spectrum of $H_\tau$, it is clear that $c(H)$ cannot be larger than the maximal point on a non-horizontal line. After perturbing $\{H_\tau\}$ into non-degenerate Hamiltonians, their spectral invariants lie in the $n$-spectrum. Repeating the arguments above for the $n$-spectrum we see that $c(H)$ is not greater than the maximal point on a non-horizontal line in the diagram corresponding to the $n$-spectrum. \\
	
	\noindent\underline{Step 2:}
	In this step we show that, for every $\tau\in(0,1]$, the action of every index $-n$ capped orbit of a non-degenerate perturbation of $H_\tau$ that appears near $s=\tau(1-\delta)$ is not greater than $C(U)-\delta\cdot \min\spec(\partial U)$. Together with the previous step, this will imply that $c(H)<C(U)$ as required.  
	Let $(\gamma,u)$ be a 1-periodic orbit of $H_\tau$ such that one of the non-degenerate capped orbits, $(\gamma',u')$, appearing after perturbing $H_\tau$ into a non-degenerate Hamiltonian, is of CZ index $-n$. In what follows we compute the Chern class of $A\in\pi_2(M)$ for which $u=u_0\#A$, where $u_0$ is a capping disk of $\gamma$ that is contained in the level set $\{s=s(\gamma)\}$. Since the manifold $M$ is monotone, we will use this computation when estimating the action of $(\gamma,u)$. 
	As mentioned above, $\hat{\gamma}(t):=\psi^{-\log s}\gamma(t/\chi'(s))\subset\partial U$ is a Reeb orbit of action $-\chi_\tau'(s)\in \spec(\partial U)$. Let $\hat u_0\subset \partial U$ be a capping disk of $\hat \gamma$ whose image coincides with $\psi^{-\log s}u_0$. Due to our assumption that $\chi''(s)>0$ for every level $s$ that contains 1-periodic orbits, 
	Lemma~\ref{lem:radial_Hamiltonians2} guarantees that 
	\begin{equation}\label{eq:unif_bound-RS_vs_CZR}
	\indRS(\gamma,u_0)= -\indCZ_R(\hat\gamma,\hat u_0) +\frac{1}{2}.
	\end{equation}
	Having a non-degenerate orbit  $(\gamma',u')$ of index $-n$ appearing after the perturbation, we can estimate $\indRS(\gamma,u_0)$ using inequality (\ref{eq:RS_after_pert}), which states that the index difference between the perturbed and original orbits is bounded by $\frac{1}{2}\dim\ker (d\varphi_H^1(\gamma(0))-\id)$. The second assertion of Lemma~\ref{lem:radial_Hamiltonians2} states that the kernel of $d\varphi_H^1(\gamma(0))-\id$ is 1-dimensional if $\chi''(s)\neq0$. Therefore, $|\indCZ_H(\gamma',u_0')-\indRS(\gamma,u_0)|\leq \frac{1}{2}$, where $u_0'$ is a capping disk of $\gamma'$ that is contained in a small neighborhood of $u_0$.  We conclude that 
	$$
	-n=\indCZ_H(\gamma',u') = \indCZ_H(\gamma',u_0')+2c_1(A)\leq \indRS(\gamma,u_0)+\frac{1}{2}+2c_1(A),
	$$
	and, together with (\ref{eq:unif_bound-RS_vs_CZR}), this yields that
	$$
	c_1(A)\geq \frac{\indCZ_R(\hat\gamma,\hat u_0)-n-1}{2}.
	$$
	Having a lower bound for the Chern class of $A$, it remains to use the monotonicity of $M$ to bound the action of the index $-n$ capped orbit $(\gamma', u')$. Since the action $(\gamma', u')$ is close to $\cA_{H_\tau}(\gamma,u)$, we will estimate the latter. Recalling the formula for the action of orbits of radial Hamiltonians, which was established in Corollary~\ref{cor:action_radial_Hamiltonian}, we have
	\begin{eqnarray}
	\cA_{H_\tau}(\gamma,u)&=& \chi_\tau(s)-s\cdot\chi_\tau'(s)-\omega(A) 
	\approx 0-\tau(1-\delta)\cdot \chi_\tau'(s) -\omega(A) \nonumber\\
	&=& \tau(1-\delta)\cdot\int_{\hat \gamma}\lambda -\omega(A) \leq (1-\delta)\cdot\int_{\hat \gamma}\lambda -\omega(A) \nonumber\\
	&=& (1-\delta)\cdot\int_{\hat \gamma}\lambda -\kappa\cdot c_1(A)\nonumber\\
	&{\leq}& (1-\delta)\cdot\int_{\hat \gamma }\lambda -\kappa\cdot \left(\frac{\indCZ_R(\hat\gamma,\hat u_0)-n-1}{2}\right).\nonumber
	\end{eqnarray}
	Our assumption that $\partial U$ is dynamically convex implies that  $\indCZ_R(\hat\gamma ,\hat u_0)\geq n+1$. Since $C(U)$ is assumed to be not-greater than $\kappa$, this yields   
	\begin{equation*}
	\cA_{H_\tau}(\gamma,u)
	\leq (1-\delta)\int_{\hat \gamma }\lambda -C( U)\cdot \left(\frac{\indCZ_R(\hat\gamma,\hat u_0)-n-1}{2}\right).
	\end{equation*}
	The next step is to bound the contact CZ index by the Reeb action divided by $C(U)$. Recall that $C(U)$ was defined as the supremum over ratios ${2\int_{\hat \gamma}\lambda}/\left({\indCZ_R(\hat \gamma,\hat u_0)-n+1}\right)$. Therefore, $\indCZ_R(\hat\gamma,\hat u_0)-n-1\geq {2\int_{\hat \gamma}\lambda}/{C(U)}-2$, and we can bound the action by 
	\begin{eqnarray}
	\cA_{H_\tau}(\gamma,u)
	&\leq& (1-\delta)\int_{\hat \gamma }\lambda -C( U)\cdot \left(\frac{\indCZ_R(\hat\gamma,\hat u_0)-n-1}{2}\right)\nonumber\\  
	&\leq& (1-\delta)\int_{\hat \gamma }\lambda - C( U)\cdot\frac{1}{2}\cdot \left(\frac{2\int_{\hat \gamma}\lambda}{C( U)}-2\right) \nonumber\\
	&=& (1-\delta)\int_{\hat \gamma }\lambda - \int_{\hat \gamma}\lambda + C( U) \leq C( U)-\delta \min\spec(\partial U). \nonumber
	\end{eqnarray}
	This finishes the proof, as the first step guaranteed that the spectral invariant $c(H)$ is not-greater than the maximal action of such a capped orbit. In particular, we conclude that $c(H)<C(U)$ as required.
\end{proof}

\subsection{Negatively monotone manifolds.}
In \cite[Proposition 5.4]{polterovich2014symplectic} Polterovich proved that on an aspherical manifold $M$, the spectral invariant of every Hamiltonian  supported in a disjoint union $U$ of portable Liouville domains is bounded by the portability number $p(U)$. Theorem~\ref{thm:spec_bound_neg} asserts that this holds true for negatively monotone manifolds, if one demands in addition that $\partial U$ is dynamically convex. As before, the Liouville coordinate of the disjoint union is given be the Liouville coordinate on each connected component.

\begin{proof}[Proof of Theorem~\ref{thm:spec_bound_neg}]
	Following \cite{polterovich2014symplectic}, the idea is to use {\it symplectic contraction} and follow the spectral invariant in the corresponding bifurcation diagram. The symplectic contraction of a Hamiltonian $H$ supported in the Liouville domain $U$ is defined to be 
	\begin{equation*}
	H_\tau:=\begin{cases}
	\tau\cdot H\circ \psi^{-\log \tau} & \text{ on } \psi^{\log \tau}U,\\
	0 & \text{ on } M\setminus\psi^{\log \tau}U
	\end{cases}
	\end{equation*}
	for $\tau\in[0,1]$, see Figure~\ref{fig:unif_bnd_neg_deformation} for an illustration.
	\begin{figure}
		\centering
		\includegraphics[scale=0.7]{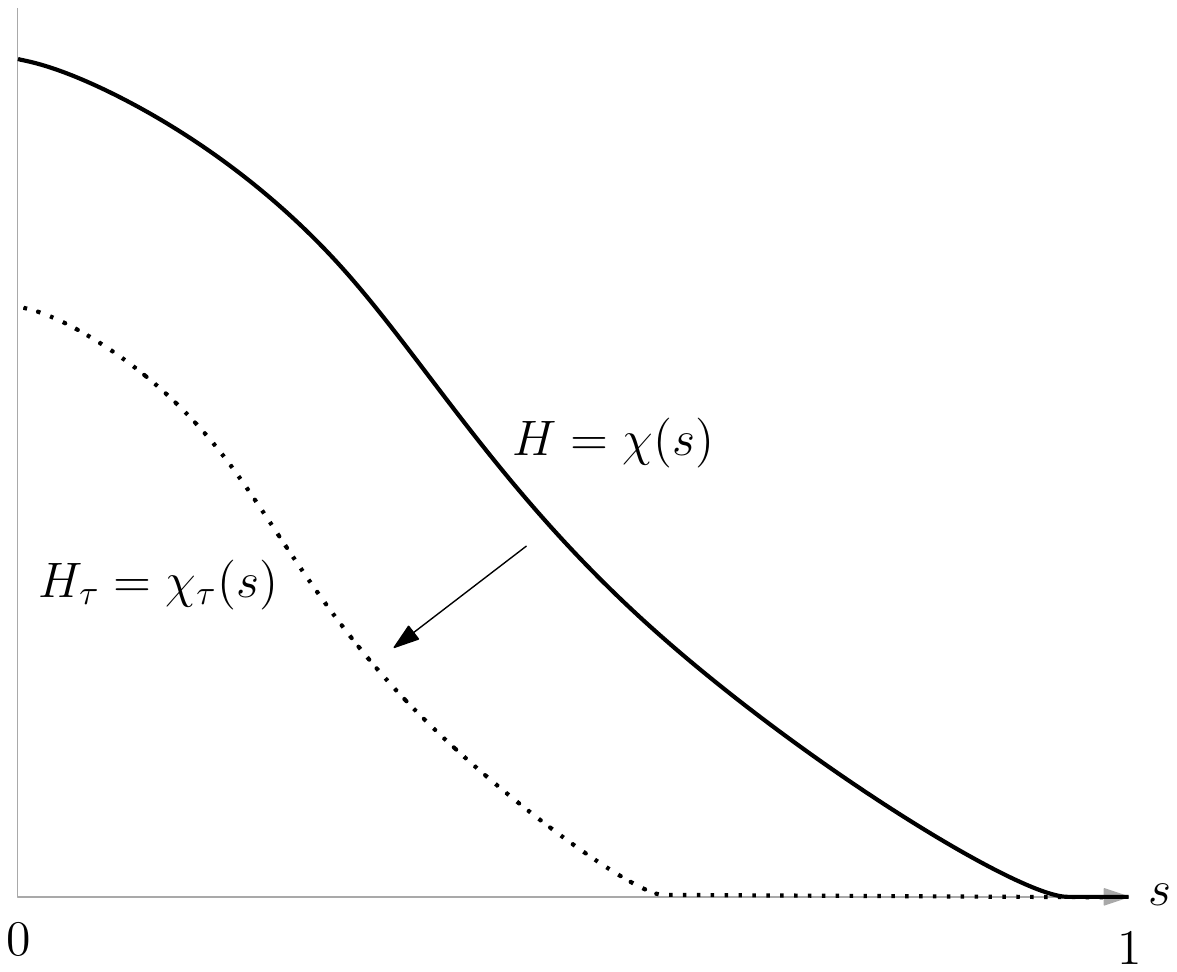}
		\caption{\small{Symplectic contraction of a radial non-increasing Hamiltonian $H$.}}
		\label{fig:unif_bnd_neg_deformation}
	\end{figure}
	The flow of the contracted Hamiltonian $H_\tau$ is the composition of the Liouville flow on flow of $H$, namely, $\varphi_{H_\tau}^t=\psi^{\log \tau}\varphi_H^t$. 
	In the aspherical case there are no intersections in the bifurcation diagram corresponding to this deformation, and the spectral invariant changes linearly, $c(H_\tau) = \tau\cdot c(H)$. However, on monotone manifolds, there could be a lot of intersections and the change of $c(H_\tau)$ is in general more complicated. To simplify the situation, we consider only non-increasing radial Hamiltonians. The claim for general Hamiltonians will follow from the monotonicity property of spectral invariants. Let $H:=\chi(s)$ where $s$ is the Liouville coordinate on $U$ and $\chi:\R\rightarrow\R$ is a non-increasing function that vanishes on $s\geq 1$ and is constant near $s=0$. We assume that the second derivative of $\chi$ does not vanish on level sets containing 1-periodic orbits, i.e., whenever $|\chi'(s)|\in\spec(\partial U)$. Let us compute the change in the action spectrum when symplectically contracting $H$. The non-constant 1-periodic orbits of $H$ and $H_\tau$ are in bijection: for $\gamma\in\cP(H)$, its image $\psi^{\log \tau}\gamma$ under the Liouville flow is a 1-periodic orbit of $H_\tau$. Given a capping disk $u_0\subset \{s=s(\gamma)\}$ of $\gamma$, its image $\psi^{\log \tau}u_0$ is a capping disk of $\psi^{\log \tau}\gamma$. The action of $(\psi^{\log \tau}\gamma,\psi^{\log \tau}u_0)$ with respect to $H_\tau$ is
	\begin{eqnarray*}
	\cA_{H_\tau}(\psi^{\log \tau}\gamma,\psi^{\log \tau}u_0) &=& H_\tau(\psi^{\log \tau}\gamma) - \int\left(\psi^{\log \tau}u_0\right)^*\omega\\ 
	&=&\tau\cdot H(\gamma) - \tau\cdot\int u_0^*\omega 
	=\tau\cdot\cA_H(\gamma,u_0).
	\end{eqnarray*}
	In addition, the RS indices of these capped orbits are equal. Indeed, given a trivialization $T:u_0^*TM\rightarrow D^2\times\R^{2n}$, we can define a trivialization along $\psi^{\log\tau} u_0$ by $T\circ d\psi^{-\log\tau}$. Under these trivializations, the differentials of the flows coincide and  thus they have the same RS index.   
	More generally, if $u$ is any capping of $\gamma$, we may write $u=u_0\#A$ for $A\in\pi_2(M)$ and then $(\psi^{\log\tau}u_0)\#A$ is a capping of $\psi^{\log \tau}\gamma$. In that case,
	\begin{eqnarray*}
	\cA_{H_\tau}(\psi^{\log \tau}\gamma,(\psi^{\log \tau}u_0)\#A) &=& \cA_{H_\tau}(\psi^{\log \tau}\gamma,\psi^{\log \tau}u_0) -\omega(A)\\
	& = &\tau\cdot\cA_H(\gamma,u_0)-\omega(A).
	\end{eqnarray*}
	If follows that the corresponding bifurcation diagram consists of lines whose slopes equal to $\cA_H(\gamma,u_0)$ for $\gamma\in\cP(H)$ and  $u_0\subset\{s=s(\gamma)\}$ and whose starting points (that is, the values at $\tau=0$) are $-\omega(A)$.	
	After perturbing $H$ into a non-degenerate Hamiltonian, it is possible to choose non-degenerate perturbations of $H_\tau$ such that 1-periodic orbits will still be in bijection with the same relations for the actions and indices. After this perturbation the spectral invariant will lie in the spectrum of index $-n$ capped orbits.
	Our next goal is to show that the slopes and the starting points are all non-negative, as illustrated in  Figure~\ref{fig:bifurcation_diag_neg_bound_c}. 
	\begin{figure}
		\centering
		\includegraphics[scale=0.5]{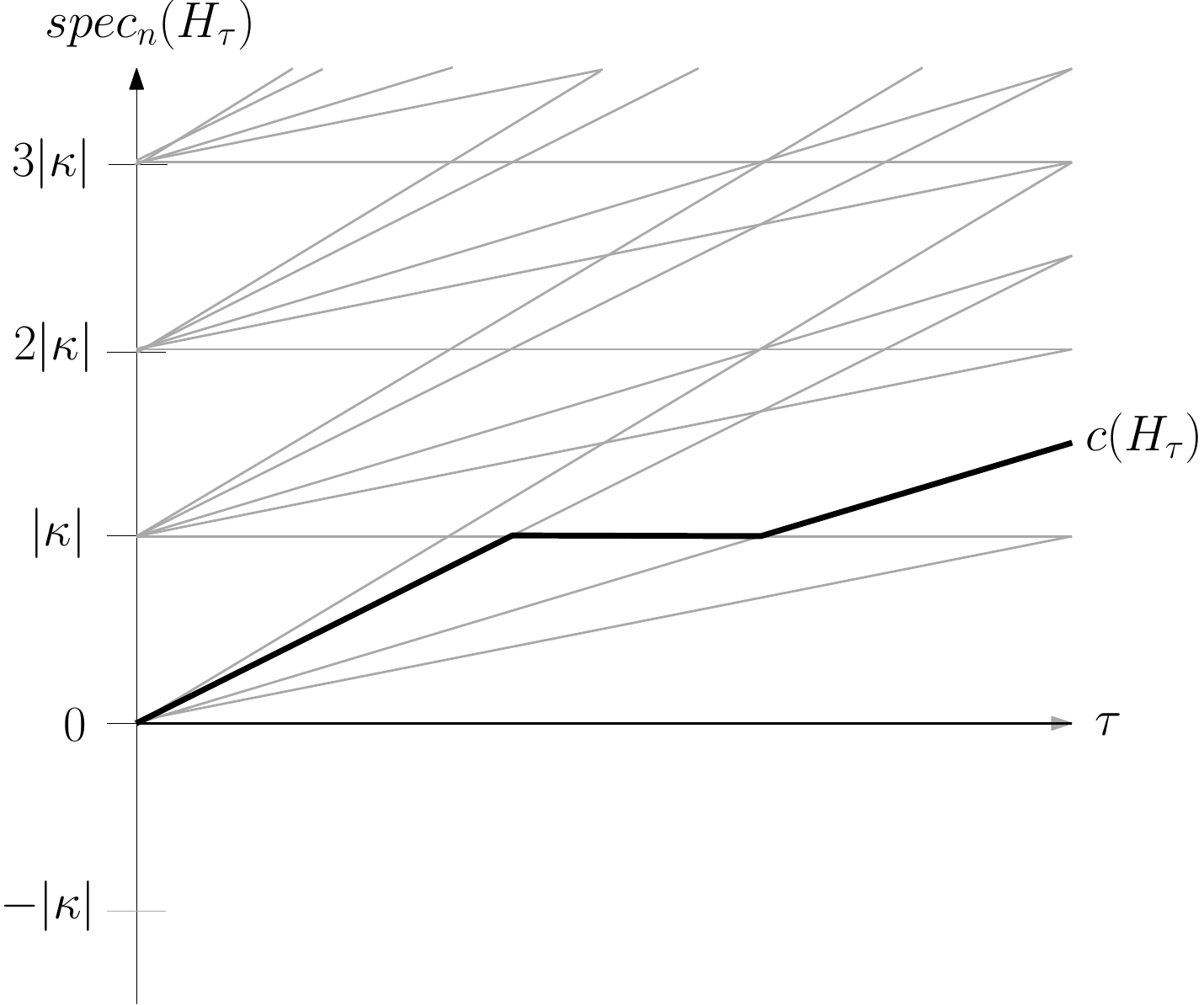}
		\caption{\small{An illustration of the bifurcation diagram for symplectic contraction of a non-increasing radial Hamiltonian on a dynamically convex domain. When all slopes and starting points are non-negative, the spectral invariant $c(H_\tau)$ can only move along lines of slope not greater than the initial slope.}}
		\label{fig:bifurcation_diag_neg_bound_c}
	\end{figure}
	This will allow us to estimate the path of the spectral invariant in the bifurcation diagram. We start with the slopes, which corresponds to actions of orbits $\gamma\in \cP(H)$ with capping disks $u_0$ that are contained in the Liouville level set:
	\begin{itemize}
		\item \underline{$\cA_H(\gamma,u_0)\geq 0$:} 
		Since we assumed $\chi$ to be non-increasing, the non-negativity of the slopes in the bifurcation diagram follows immediately from the formula for actions of orbits of radial Hamiltonians, that was stated in Corollary~\ref{cor:action_radial_Hamiltonian}, when applied for $u=u_0$, $A=0$:
		\begin{equation*}
			\cA_H(\gamma,u_0)=\chi(s)-s\cdot\chi'(s)\geq 0.
		\end{equation*}
		
		\item \underline{$-\omega(A)\geq 0$:} The non-negativity of the stating points requires index computations. 
		We start by recalling that the index difference between the perturbed and original orbits is bounded by half the dimension of $\ker(d\varphi_H^1(\gamma(0))-\id)$, as stated in  (\ref{eq:RS_after_pert}).
		Due to our choice of $\chi$, its second derivative $\chi''(s)$ does not vanish whenever $-\chi'(s)\in\spec(\partial U)$ and Lemma~\ref{lem:radial_Hamiltonians2} guarantees that the kernel of $d\varphi_H^1(\gamma(0))-\id$ is 1-dimensional. Thus
		\begin{equation*}
		|\indCZ_H(\gamma',u_0') - \indRS(\gamma,u_0)|\leq \frac{1}{2},
		\end{equation*}
		where $u_0'$ is a capping disk of $\gamma'$, obtained as a perturbation  of $u_0\subset \{s=s(\gamma)\}$. We therefore conclude that capped orbits $(\gamma',u')$ of index $-n$ can appear only for $(\gamma,u)$ such that 
		\begin{eqnarray}\label{eq:neg_bnd_RS}
		\frac{1}{2}&\geq& |\indRS(\gamma,u_0)-\indCZ_H(\gamma',u_0')| = |\indRS(\gamma,u_0)-\indCZ_H(\gamma',u')+2c_1(A)|\nonumber\\
		&=& |\indRS(\gamma,u_0)+n+2c_1(A)|,
		\end{eqnarray}
		where $A\in\pi_2(M)$ is such that $u'=u_0'\#A$.
		Next, we use Lemma~\ref{lem:radial_Hamiltonians2} in order to replace the RS index of $(\gamma,u_0)$ by the contact CZ index of the corresponding capped Reeb orbit $(\hat \gamma,\hat u_0)$. Here, $\hat\gamma(t):=\psi^{-\log s}\gamma(t/\chi'(s))\in \cP(\partial U)$ and $\hat u_0$ is a capping disk whose image coincides with that of $\psi^{-\log s} u_0$.  
		Recalling that $\chi'(s)\leq 0$, Lemma~\ref{lem:radial_Hamiltonians2} states that the RS index of the capped 1-periodic orbit $(\gamma,u_0)$ of the radial Hamiltonian $H$ is given by 
		\begin{equation*}
		\indRS(\gamma,u_0) = -\indCZ_R(\hat \gamma,\hat u_0) +\frac{1}{2}\sign(\chi''(s))\leq  -\indCZ_R(\hat \gamma,\hat u_0) +\frac{1}{2}. 
		\end{equation*}
		Therefore, capped 1-periodic orbits of CZ index $-n$ can appear only for $(\gamma,u)$ with $u=u_0\#A$ such that
		\begin{equation*}
		c_1(A) \geq \frac{-\indRS(\gamma,u_0)-n-\frac{1}{2}}{2} \geq \frac{\indCZ_R(\hat \gamma,\hat u_0) -n-1}{2},
		\end{equation*}
		for some Reeb orbit $\hat \gamma\in\cP(\partial U)$ with a capping	 $\hat u_0\subset \partial U$.
		Since $\partial U$ is dynamically convex, $\indCZ_R(\hat \gamma,\hat u_0)\geq n+1$ for every capped Reeb orbit. As a consequence, $c_1(A)\geq 0$ and, recalling that $M$ is negatively monotone, $-\omega(A)=-\kappa c_1(A)\geq 0$ as required.
%
\end{itemize}
	Having proved that the spectral invariant lies in a diagram consisting of  lines with non-negative slopes and non-negative starting points, let us explain how the upper bound $c(H)\leq p(U)$ follows.
	We follow $c(H_\tau)$ along the diagram as $\tau$ grows from $0$ to $1$. When $\tau=0$, $H_0=0$ and by the stability property of spectral invariants, $c(H_0)=0$. When $\tau>0$ is very small, there are no intersections in the diagram yet. As a result, $c(H_\tau)$ moves along a single line, $\ell$, that starts at zero and whose slope we denote by $a\geq 0$: $c(H_\tau)=a\cdot \tau$ for $\tau<\epsilon$. We claim that $c(H_\tau)\leq a\cdot \tau$ for all $\tau\leq 1$. Indeed, since the line $\ell$ starts at the lowest possible point (zero), every line $\ell'$ of slope bigger than $a$ is completely contained in the upper region bounded by $\ell$, namely $\ell'\subset \{y>a\cdot\tau\}$. Thus, when starting on the line $\ell$, the spectral invariant $c(H_\tau)$ can only move along lines of slope $\leq a$, see Figure~\ref{fig:bifurcation_diag_neg_bound_c} for an illustration. Therefore, $c(H_\tau)\leq a\cdot\tau$ for all $\tau\leq 1$. To conclude a bound for the spectral invariant of $H=H_1$, it remains to bound the slope $a$ of $\ell$. This argument is identical to that in the aspherical case, presented by Polterovich in \cite{polterovich2014symplectic}: The contracted Hamiltonian $H_\tau$ is supported in $\psi^{\log \tau} U$. When $\tau$ is very small, this set is displaceable in $U$ with energy $e(\psi^{\log \tau} U;U)$. By the energy-capacity inequality, $c(H_\tau)\leq e(\psi^{\log \tau} U;M)\leq e(\psi^{\log \tau} U;U)$ for $\tau$ close enough to zero. Since $c(H_\tau)$ lies on $\ell$ there, we conclude that the slope is bounded by the ratio $a\leq  e(\psi^{\log \tau} U;U)/\tau$. Taking the limit $\tau\rightarrow0$ we obtain $c(H)\leq a\cdot 1 \leq p(U)$.
\end{proof}

\section{Estimating the invariant $C(U)$ in special cases.}\label{sec:the_invariant_C}
In this section we provide upper bounds for the invariant $C$ from Definition~\ref{def:ratio_invariant} on several classes of domains. The domains considered here are all topological balls, which means that their boundaries are aspherical, $\pi_2(S^{2n-1})=0$. As a consequence, the contact CZ index is independent of the choice of a capping disk, and we will use the notation $\indCZ_R(\gamma)$.
Let us start with the simplest example, a generic ellipsoid. 
\begin{exam}\label{exa:C_for_ellipsoid}
	Let $U=E(a_1,\dots,a_n)=\{z\in\C^{n}:\sum_{j=1}^n \frac{\pi|z_j|^2}{a_j} \leq 1\}$ be an ellipsoid such that $a_i/a_j$ is irrational when $i\neq j$. The periodic Reeb orbits on $\partial E$ are $\gamma_{k,\ell}:=\{e^{2\pi i t/a_k} z_k\}_{t\in[0,a_k\ell]}$ for $k=1,\dots,n$ and $\ell\in\N_{>0}$. The action (or period) of $\gamma_{k,\ell}$ is $a_k\ell$ and its CZ index (with respect to any capping disk $u_0\subset \partial E$) is 
	$$
	CZ_R(\gamma_{k,\ell})=n-1+2\sum_{j=1}^n\Big\lfloor\ell\cdot\frac{a_k}{a_j} \Big\rfloor,
	$$ 
	see e.g. \cite[p.16]{gutt2018symplectic}. A simple computation shows that in this case $C(E) = \min_j a_j=c_G(E)$. 
\end{exam}
 
Wider classes of examples that are considered below are convex and concave toric domains.  Since the Reeb flow on toric domains is degenerate (unless the domain is an ellipsoid), we need to extend Definition~\ref{def:ratio_invariant} of the invariant $C(U)$ to degenerate domains.
\begin{defin}\label{def:ratio_invariant_degenerate}
	Let $U\subset M$ be a domain with an incompressible contact type boundary.
	\begin{itemize}
		\item Assume that the Reeb flow on $\partial U$ is non-degenerate. For every $T>0$, define
		\begin{equation*}
		C_T(U) := \sup\left\{\frac{2\int_\gamma\lambda}{\indCZ_R(\gamma,u_0)-n+1}\ :\ \gamma\in\cP(U),\ \int_\gamma\lambda\leq T,\ u_0\subset \partial U\right\},
		\end{equation*}
		and notice that $C(U)=\sup_{T>0}C_T(U)$.
		\item When the Reeb flow on $\partial U$ is degenerate, we define
		\begin{eqnarray}
		\nonumber 
		C_T(U):=\liminf_{U'\rightarrow U} C_T(U'),\quad C(U):=\sup_{T>0} C_T(U),
		\end{eqnarray}
		where the limit is over domains $U'$ with incompressible contact type boundaries and non-degenerate Reeb flows that $C^1$-converge to $U$. 
	\end{itemize}   
\end{defin}
 
\subsection{Convex and concave toric domains.}
Denote by $\R^n_{\geq0}$ the set of points $x\in\R^n$ with non-negative entries, $x_i\geq 0$. Consider the moment map $\mu:\C^n\rightarrow\R^n_{\geq 0}$ defined by $\mu(z_1,\dots,z_n)=\pi(|z_1|^2,\dots,|z_n|^2)$. We say that  a domain $U\subset M$ with a contact type boundary is a  {\it toric domain} if it is symplectomorphic to $X_\Omega:=\mu^{-1}(\Omega)\subset \C^n$, for some domain $\Omega\subset \R^n_{\geq0}$ (that is, a connected set that is open in the topology of $\R^n_{\geq0}$). Moreover,
\begin{itemize}
 	 \item If $\hat \Omega:=\left\{(x_1,\dots, x_n)\in \R^n: (|x_1|,\dots,|x_n|)\in \Omega \right\}$ is convex, we say that $X_\Omega$ is a {\it convex toric domain}.
	\item If $\bar \Omega$ is compact and $\R^n_{\geq 0}\setminus\Omega$ is convex, we say that $X_\Omega$ is a {\it concave toric domain}.
\end{itemize}
In \cite{gutt2018symplectic}, Gutt and Hutchings estimate actions and Conley-Zehnder indices of periodic Reeb orbits for convex and concave toric domains, in order to compute the capacities $c_k$ coming from positive $S^1$-equivariant symplectic homology. Using their calculations, we give upper bounds for the invariant $C(U)$ from Definition~\ref{def:ratio_invariant} when $U$ is a concave or convex toric domain.
\begin{claim}\label{clm:concave_toric_C}
	Suppose that $U\subset M$ is a domain with an incompressible contact type boundary that is symplectomorphic to a concave toric domain. Then, $C(U)=c_G(U)$.
\end{claim}
\begin{claim}\label{clm:convex_toric_C}
	Suppose that $U\subset M$ is a domain with an incompressible contact type boundary that is symplectomorphic to a convex toric domain, $U\cong X_\Omega$. Then, $c_G(U)\leq C(U)\leq c_G(B)$, for every ball $B$ that contains $X_\Omega$.
\end{claim}

The lower bounds in the above claims actually hold for general nice star-shaped domains, and follow from a comparison between the invariant $C(U)$ and certain invariants defined by Gutt and Hutchings in \cite{gutt2018symplectic}.
\begin{lemma}\label{lem:C_vs_ck}
	Let $U\subset \R^{2n}$ be a nice star shaped domain, then
	$
	C(U) \geq c_G(U).
	$
\end{lemma}
\begin{proof}
	The proof relies on the capacities $c_k$ coming from positive $S^1$-equivariant symplectic homology\footnote{Alternatively, one can use the capacity coming from symplectic homology \cite{floer1994applications}. Similar arguments imply that $C(U)$ is greater or equal to this capacity as well.}, which were introduced by Gutt and Hutchings in \cite{gutt2018symplectic}. These are numerical invariants of nice star-shaped domains which admit several useful properties. Let us state the properties that will be of use for us (see \cite[Theorem 1.1]{gutt2018symplectic}):
	\begin{enumerate}
		\item \label{itm:prop_1_GH_cap} $c_1(U)\geq c_G(U)$ for every nice star-shaped domain $U$.
		\item $c_k$ are continuous with respect to the Hausdorff metric.
		\item For a non-degenerate nice star-shaped domain $U$ and $k\in \N$, $c_k(U)$ belongs to the action spectrum of Reeb orbits with contact Conley-Zehnder index equal to $2k+n-1$.
	\end{enumerate} 
	We remark that property \ref{itm:prop_1_GH_cap} follows directly from the monotonicity and normalization properties stated in  \cite[Theorem 1.1]{gutt2018symplectic}. 
	We conclude that for every non-degenerate nice star-shaped domain,
	$C(U)\geq \sup_k c_k(U)/k \geq c_1(U)\geq c_G(U)$. For degenerate domains, the inequality $C(U)\geq  c_G(U)$ now follows from the fact that $c_G$ is continuous in the Hausdorff metric.
\end{proof}

Having the established the lower bound, it remains to bound the invariant $C$ from above. Let us start with convex toric domains.
\begin{proof}[Proof of Claim~\ref{clm:convex_toric_C}]
	By Definition~\ref{def:ratio_invariant_degenerate}, it is enough to prove the upper bound for $C_T(U)$, for arbitrary $T>0$.
	In \cite[p.18-20]{gutt2018symplectic}, Gutt and Hutchings show that one can perturb any convex toric domain $U$ into a domain with a contact type boundary and non-degenerate Reeb flow, such that, after the perturbation, every Reeb orbit $\gamma\in\cP(\partial U)$
	whose action is not greater than $T$ corresponds to a vector with non-negative integer entries $v\in\N^n$ (here $\N$ denotes the set of natural numbers including zero) satisfying 
	\begin{equation}\label{eq:GH_action_index-convex}
	\int_\gamma\lambda  \approx \|v\|_\Omega^* \quad \text{and}\quad 
	Z(v)+2\sum_{i=1}^n v_i \leq \indCZ_R(\gamma) \leq n-1+ 2\sum_{i=1}^n v_i,
	\end{equation}
	where $\|v\|_\Omega^*:= \sup\{\left<v,w\right>:w\in\Omega\}$ is the support function associated to $\Omega$ and $Z(v)$ is the number of elements of $v$ that are equal to zero.
	As a consequence,
	\begin{equation}\label{eq:C_T_using_v}
	C_T(U) \leq \sup \left\{\frac{\|v\|_\Omega^*+\varepsilon}{\sum_{i=1}^{n}v_i -\frac{n-1-Z(v)}{2}} \ \Big|\  v\in\N^n\right\},
	\end{equation}
	where $\varepsilon>0$ is arbitrarily small.
	Let us show that the right hand side of (\ref{eq:C_T_using_v}) is bounded, up to $\varepsilon$, by $\sup_{w\in\Omega} \|w\|_{\infty}$. Indeed,
	\begin{eqnarray}
	\frac{\|v\|_\Omega^*}{\sum_{i=1}^{n}v_i -\frac{n-1-Z(v)}{2}}&=& \sup_{w\in\Omega} \frac{\left<v,w\right>}{\sum_{i=1}^{n}v_i -\frac{n-1-Z(v)}{2}} = \sup_{w\in\Omega} \frac{\sum_{i=1}^n v_iw_i}{\sum_{i=1}^{n}v_i -\frac{n-1-Z(v)}{2}}
	\nonumber\\
	&\leq& \sup_{w\in\Omega}\  (\max_{i} w_i)\cdot  \frac{\sum_{i=1}^n v_i}{\sum_{i=1}^{n}v_i -\frac{n-1-Z(v)}{2}} \leq \sup_{w\in\Omega} \|w\|_\infty.
	\label{eq:C_T_bdd_by_infty_norm}
	\end{eqnarray}
	Let $B=B(a)\subset \R^{2n}$ be a ball of capacity $a$ that contains $X_\Omega$, then $\mu(B)\supset \Omega$. As a consequence, $\sup_{w\in\Omega} \|w\|_\infty \leq \sup_{w\in\mu(B)} \|w\|_\infty = a$.
\end{proof}
\begin{rem}
	In fact we proved a stronger upper bound for the invariant $C(U)$ than the one stated in Claim~\ref{clm:convex_toric_C}. Inequality (\ref{eq:C_T_bdd_by_infty_norm}) implies that $C(X_\Omega)\leq \sup_{w\in\Omega} \|w\|_\infty$. In particular, the invariant $C(X_\Omega)$ can be bounded by the minimal $a$ such that the polydisk $D(a)^n$ contains $X_\Omega$.
\end{rem}

\begin{figure}
	\centering
	\includegraphics[scale=0.5]{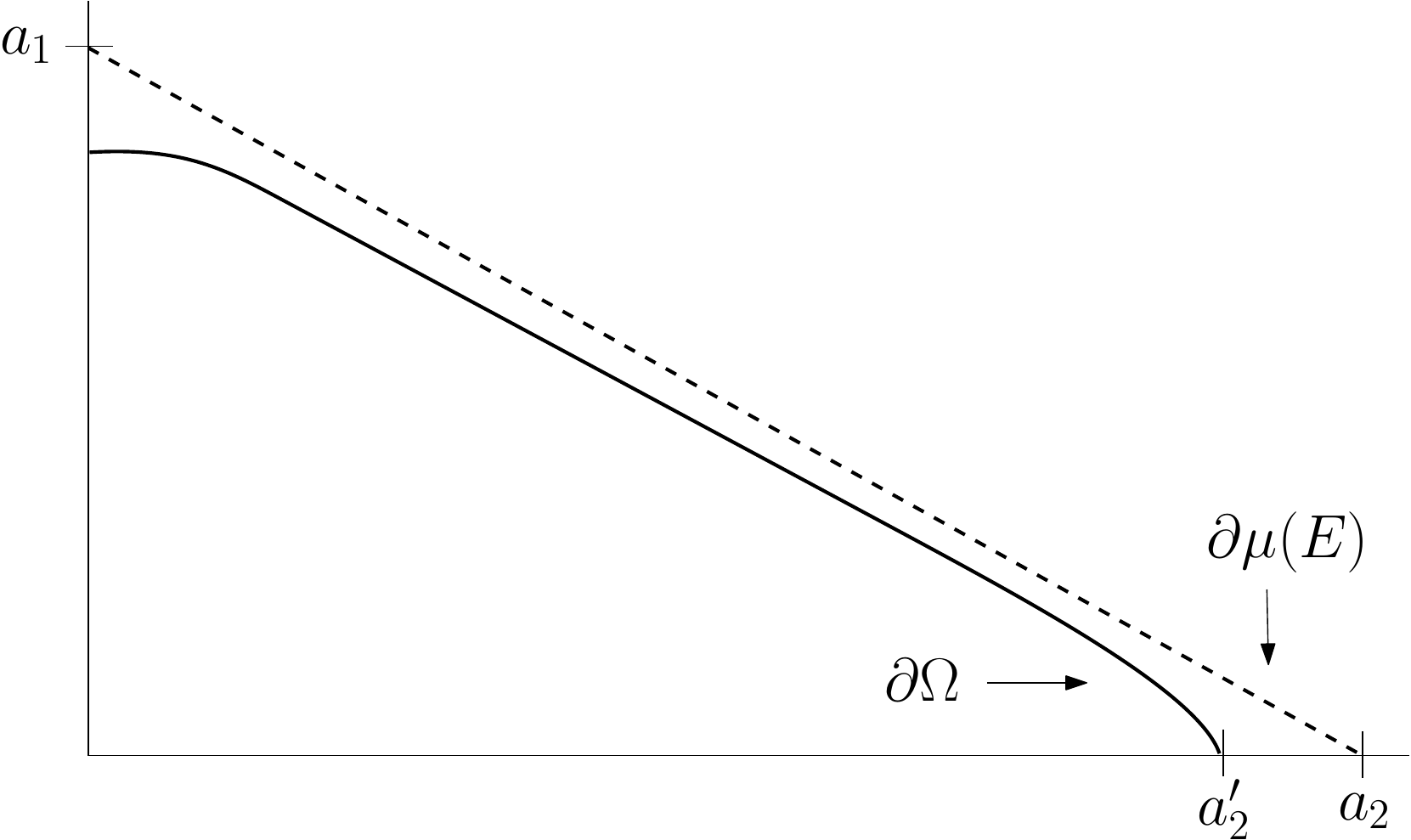}
	\caption{\small{The dashed line is the image of the boundary of $E$ under the moment map and the solid line is $\partial\Omega$, which intersects the axes at almost right angles.}}
	\label{fig:C_not_cont}
\end{figure}

The following example shows that invariant $C$ is not continuous with respect to the Hausdorff metric and is not monotone with respect to inclusion. 
\begin{example}\label{exa:C_not_cont}
	Given an ellipsoid $E=E(a_2,a_1)\subset \C^2$ with $a_1<a_2$, then it is a convex toric domain. As mentioned in the proof of Claim~\ref{clm:convex_toric_C}, Gutt and Hutchings showed in \cite[p.18-20]{gutt2018symplectic} that there exists a perturbation of $E(a_2,a_1)$ for which the periodic Reeb orbit correspond to vectors $v\in \N^2$ whose actions and indices are given by (\ref{eq:GH_action_index-convex}). Let us describe this perturbation. Let $X_\Omega$ be a convex toric domain such that $\Omega$ is Hausdorff-close to $\mu(E)$, and such that the curve $\partial \Omega$ is almost perpendicular to the axes, see Figure~\ref{fig:C_not_cont}.
	Denote by $a_2'$ the intersection point of $\partial \Omega$ with the $x$-axis and note that $a_2'$ is close to $a_2$ and, in particular, is greater than $a_1$. The next step is to perturb $X_\Omega$ into a non-degenerate domain. After this perturbation there exists a periodic Reeb orbit $\gamma$ corresponding to $v=(1,0)\in\N^2$. Recalling (\ref{eq:GH_action_index-convex}), the action of $\gamma$ is approximately $\|v\|_{\Omega}^*=a_2'$ and its index is $\indCZ_R(\gamma)=1+2=3$.
	As a result, we obtain a domain $U$ with a contact type boundary and non-degenerate flow, that is contained in $E$ and is Hausdorff-close to it, such that  
	$$
	C(U)\geq \frac{2\int_\gamma\lambda}{\indCZ_R(\gamma)-n+1} \approx\frac{2a_2'}{3-2+1} = a_2'\approx a_2.
	$$ 
	By Example~\ref{exa:C_for_ellipsoid}, $C(E)=a_1$ and we conclude that $C$ is neither monotone nor  Hausdorff-continuous.
\end{example}

We turn to prove Claim~\ref{clm:concave_toric_C}.
\begin{proof}[Proof of Claim~\ref{clm:concave_toric_C}]
	By Definition~\ref{def:ratio_invariant_degenerate}, it is enough to prove the upper bound for $C_T(U)$, for arbitrary $T>0$. 
	In \cite[p.22-23]{gutt2018symplectic}, Gutt and Hutchings show that one can perturb $U$ into a domain with a contact type boundary and non-degenerate flow, such that, after the perturbation, every Reeb orbit  $\gamma\in\cP(\partial U)$ either has a large contact CZ index (in which case, either the action is larger than $T$, or the ratio in the definition of $C(U)$ is small) or corresponds to a vector $v\in\N^n_{>0}$, such that 
	\begin{equation*}
	\int_\gamma\lambda  \leq [v]_\Omega+\varepsilon \quad \text{and}\quad 
	1-n+2\sum_{i=1}^n v_i \leq \indCZ_R(\gamma) \leq 2\sum_{i=1}^n v_i,
	\end{equation*}
	where $[v]_\Omega:=\inf\{\left<v,w\right>:w\in\R_{\geq0}^n\setminus\Omega\}$ and $\varepsilon>0$ is arbitrarily small. 
	As a consequence, 
	\begin{equation}
	C_T(U) \leq \sup \left\{\frac{[v]_\Omega +\varepsilon}{\sum_{i=1}^{n}v_i -n+1} \ \Big|\  v\in\N^n_{>0}\right\}.
	\end{equation}
	Fix $\hat v\in \N^n_{>0}$ such that $C_T(U)\leq \frac{[\hat v]_\Omega}{\sum_{i=1}^{n}\hat v_i -n+1}+2\varepsilon$. Consider the ellipsoid defined by 
	$$
	E:=\mu^{-1}\left(\{x\in\R^n_{\geq0}:\left<x,\hat v\right>< [\hat v]_\Omega\}\right), 
	$$
	then $E\subset U$ and $[\hat v]_{\mu(E)} = [\hat v]_\Omega$. Let $0<a_1\leq \cdots\leq a_n$ such that $E=E(a_1,\dots, a_n)$, then
	\begin{eqnarray*}
		C_T(U)&\leq&  \frac{[\hat v]_\Omega}{\sum_{i=1}^{n}\hat v_i -n+1} +2\varepsilon=
		\frac{[\hat v]_{\mu(E)}}{\sum_{i=1}^{n}\hat v_i -n+1} +2\varepsilon\\
		&=& \frac{\min_i a_i\hat v_i}{\sum_{i=1}^{n}\hat v_i -n+1} +2\varepsilon \leq \frac{\min_i a_i\hat v_i}{\max_i \hat v_i} +2\varepsilon \leq a_1+2\varepsilon \\
		& =& c_G(E) +2\varepsilon \leq c_G(U)+2\varepsilon .
	\end{eqnarray*}
\end{proof}

\subsection{Strictly convex domains: comparison to Ishikawa's constant.}
	In this section, we show that when $U$ is strictly convex, $C(U)<\infty$. This follows from computations carried by Ishikawa in \cite{ishikawa2015spectral}. On strictly convex domains, Ishikawa bounded the ratio between the action and CZ index of 1-periodic orbits of radial Hamiltonians in terms of the sectional curvatures of the boundary.
	\begin{defin}[{\cite[Definition 4.1]{ishikawa2015spectral}}] \label{def:ishikawas_invariant}
		Let $U\subset \R^{2n}$ be a strictly convex open set with a smooth boundary, such that $0\in U$. Denote by $f:\R^{2n}\rightarrow\R$ the squared semi-norm associated to $U$, namely $f(ty)=t^2$ for $y\in\partial U$. 
		Define 
		$$
		\hat C_0(U):=\inf_{V\subset \R^{2n}, a>0}\frac{2\pi}{a}
		$$
		where the infimum is taken over all one-dimensional complex subspaces $V\subset \C^n\cong \R^{2n}$ and $a>0$ such that $D^2f(x)>a|_{V\times V}\oplus 0|_{V^\perp\times V^\perp}$, for every $x\in \R^{2n}\setminus\{0\}$.
	\end{defin}  
	Roughly speaking, 
	$\hat C_0(U)$ is $2\pi$ over the maximum over all complex planes of the minimal sectional curvature in $\partial U$, restricted to the plane. In particular, it is finite for every strictly convex domain. 
	\begin{claim}
		Let $U\subset \R^{2n}$ be a strictly convex open set with a smooth boundary, such that $0\in U$. Then, $C(U)\leq \hat C_0(U)$.
	\end{claim} 
	\begin{proof}
	Let $\hat \gamma\in\cP(\partial U)$ be a periodic Reeb orbit and denote its action by $\alpha:=\int_{\hat \gamma}\lambda>0$. In what follows we show that the ratio $2\int_{\hat \gamma}\lambda/\left(\indCZ_R(\hat \gamma)-n+1\right)$ in the definition of $C(U)$ is bounded  by $\hat C_0(U)$.
	The argument goes through radial Hamiltonians in order to use a result from \cite{ishikawa2015spectral}.
	Consider a radial Hamiltonian $H$ that is linear with slope $-\alpha$ with respect to the Liouville coordinate near the boundary of $U$, i.e., $H|_{\cN(\partial U)}=-\alpha\cdot s = -\alpha\cdot f$, where $f$ is the squared semi-norm associated to $U$, as in Definition~\ref{def:ishikawas_invariant}. Lemma~\ref{lem:radial_Hamiltonians1} states that $\gamma(t)=\hat\gamma(\chi'(s)\cdot t)=\hat\gamma(-\alpha\cdot t)$ is a 1-periodic orbit of the Hamiltonian flow of $H$.
	The linearized flow $\Phi(t):=d\varphi_H^t(\gamma(0))$ of $H$ along $\gamma$ satisfies the equation 
	$$
	\frac{d}{dt}\Phi(t) = J_0 D^2H(\gamma(t))\cdot \Phi(t)= -\alpha J_0 D^2f(\gamma(t))\cdot\Phi(t).
	$$  
	Lemma 3.1, item (ii) from \cite{ishikawa2015spectral} states that, given a complex plane $V\subset \C^n$ and a path $\{\Phi(t)\}$ of symplectic matrices satisfying $\frac{d}{dt}\Phi(t) = J_0S(t)\Phi(t)$ for $S(t)\leq -c|_{V\times V}\oplus 0|_{V^\perp\times V^\perp}$,  its RS index is bounded by\footnote{Lemma 3.1 from \cite{ishikawa2015spectral} is formulated for the $\max\indCZ$, which was defined by Ishikawa on p.8. It follows from (\ref{eq:RS_after_pert}) that in our case $\max\indCZ(\gamma,u_0)=\indRS(\gamma,u_0)+1$.}
	$$
	RS(\Phi)\leq -n-2[c/2\pi]^<-1,
	$$
	where $[x]^<$ stands for the largest integer that is smaller than $x$. 
	Applying this statement to $S(t) = -\alpha D^2f (\gamma(t))$ and $c=\alpha\cdot a$, it follows that the RS index of the orbit $\gamma$, with respect a capping disk $u_0\subset \{s=s(\gamma)\}$ is bounded by
	\begin{equation}
	\indRS(\gamma,u_0)\leq-n -2\Big[\frac{\alpha\cdot a}{2\pi}\Big]^<-1\leq   -n -2\Big[\frac{\alpha}{\hat C_0(U)}\Big]^<-1.\label{eq:Ishikawa_CZ_hatC0} 
	\end{equation} 
	By Lemma~\ref{lem:radial_Hamiltonians2}, the RS index of $(\gamma,u_0)$ is equal to minus the contact CZ index of $\hat\gamma$. Therefore, inequality (\ref{eq:Ishikawa_CZ_hatC0}) yields
	\begin{equation*}
	\indCZ_R(\hat \gamma) \geq n+2\Big[\frac{\alpha}{\hat C_0(U)}\Big]^<+1= n+2\Big[\frac{\int_{\hat{\gamma}}\lambda}{\hat C_0(U)}\Big]^<+1.
	\end{equation*}
	As the above inequality holds for every Reeb orbit, we can use it to bound the invariant $C(U)$ as follows:
	\begin{eqnarray*}
		C( U) &=& \sup_{\hat\gamma} \frac{2\int_{\hat \gamma}\lambda}{ \indCZ_R(\hat \gamma,\hat u_0)-n+1}
		\leq \sup_{\hat\gamma} \frac{2\int_{\hat \gamma}\lambda}{n+2\Big[\frac{\int_{\hat{\gamma}}\lambda}{\hat C_0(U)}\Big]^<+1-n+1}\\
		&\leq& \sup_{\hat\gamma} \frac{2\int_{\hat \gamma}\lambda}{2\cdot \frac{\int_{\hat{\gamma}}\lambda}{\hat C_0(U)}-2 +2} = \hat C_0(U),
\end{eqnarray*}
	where the last inequality follows from the fact that $[x]^<\geq x-1$.
\end{proof}

\bibliographystyle{plain}
\bibliography{refs}

\begin{thebibliography}{10}

\bibitem{bourgeois2009survey}
Fr{\'e}d{\'e}ric Bourgeois.
\newblock A survey of contact homology.
\newblock {\em New perspectives and challenges in symplectic field theory},
  49:45--71, 2009.

\bibitem{buhovsky2020poisson}
Lev Buhovsky, Alexander Logunov, and Shira Tanny.
\newblock Poisson brackets of partitions of unity on surfaces.
\newblock {\em Commentarii Mathematici Helvetici}, 95(1):247--278, 2020.

\bibitem{entov2004quasi}
Michael Entov and Leonid Polterovich.
\newblock Quasi-states and symplectic intersections.
\newblock {\em Commentarii Mathematici Helvetici}, 81(1):75--99, 2006.

\bibitem{entov2009rigid}
Michael Entov and Leonid Polterovich.
\newblock Rigid subsets of symplectic manifolds.
\newblock {\em Compositio Mathematica}, 145(3):773--826, 2009.

\bibitem{floer1994applications}
Andreas Floer, Helmut Hofer, and Krzysztof Wysocki.
\newblock Applications of symplectic homology {I}.
\newblock {\em Mathematische Zeitschrift}, 217(1):577--606, 1994.

\bibitem{ganor2020floer}
Yaniv Ganor and Shira Tanny.
\newblock Floer theory of disjointly supported {H}amiltonians on symplectically
  aspherical manifolds.
\newblock {\em arXiv preprint arXiv:2005.11096}, 2020.

\bibitem{gutt2014generalized}
Jean Gutt.
\newblock Generalized {C}onley-{Z}ehnder index.
\newblock In {\em Annales de la Facult{\'e} des sciences de Toulouse:
  Math{\'e}matiques}, volume~23, pages 907--932, 2014.

\bibitem{gutt2018symplectic}
Jean Gutt and Michael Hutchings.
\newblock Symplectic capacities from positive {S}1--equivariant symplectic
  homology.
\newblock {\em Algebraic \& Geometric Topology}, 18(6):3537--3600, 2018.

\bibitem{hofer1998dynamics}
Helmut Hofer, Krzusztof Wysocki, and Eduard Zehnder.
\newblock The dynamics on three-dimensional strictly convex energy surfaces.
\newblock {\em Annals of Mathematics}, pages 197--289, 1998.

\bibitem{humiliere2016towards}
Vincent Humili{\`e}re, Fr{\'e}d{\'e}ric Le~Roux, and Sobhan Seyfaddini.
\newblock Towards a dynamical interpretation of {H}amiltonian spectral
  invariants on surfaces.
\newblock {\em Geometry \& Topology}, 20(4):2253--2334, 2016.

\bibitem{ishikawa2015spectral}
Suguru Ishikawa.
\newblock Spectral invariants of distance functions.
\newblock {\em Journal of Topology and Analysis}, page 1650025, 2015.

\bibitem{mcduff2012j}
Dusa McDuff and Dietmar Salamon.
\newblock {\em J-holomorphic curves and symplectic topology}, volume~52.
\newblock American Mathematical Soc., 2012.

\bibitem{oh2005construction}
Yong-Geun Oh.
\newblock Construction of spectral invariants of {H}amiltonian paths on closed
  symplectic manifolds.
\newblock In {\em The breadth of symplectic and Poisson geometry}, pages
  525--570. Springer, 2005.

\bibitem{payette2018geometry}
Jordan Payette.
\newblock The geometry of the {P}oisson bracket invariant on surfaces.
\newblock {\em arXiv preprint arXiv:1803.09741}, 2018.

\bibitem{polterovich2012quantum}
Leonid Polterovich.
\newblock Quantum unsharpness and symplectic rigidity.
\newblock {\em Letters in Mathematical Physics}, pages 1--20, 2012.

\bibitem{polterovich2014symplectic}
Leonid Polterovich.
\newblock Symplectic geometry of quantum noise.
\newblock {\em Communications in Mathematical Physics}, 327(2):481--519, 2014.

\bibitem{polterovich2014function}
Leonid Polterovich and Daniel Rosen.
\newblock {\em Function theory on symplectic manifolds}.
\newblock American Mathematical Society, 2014.

\bibitem{robbin1993maslov}
Joel Robbin and Dietmar Salamon.
\newblock The {M}aslov index for paths.
\newblock {\em Topology}, 32(4):827--844, 1993.

\bibitem{schwarz2000action}
Matthias Schwarz.
\newblock On the action spectrum for closed symplectically aspherical
  manifolds.
\newblock {\em Pacific Journal of Mathematics}, 193(2):419--461, 2000.

\bibitem{seyfaddini2014spectral}
Sobhan Seyfaddini.
\newblock Spectral killers and {P}oisson bracket invariants.
\newblock {\em Journal of Modern Dynamics}, 9:51--66, 2015.

\end{thebibliography}

\paragraph{Shira Tanny,}$ $\\
School of Mathematical Sciences\\ 
Tel Aviv University \\
Ramat Aviv, Tel Aviv 69978\\
Israel\\
E-mail: tanny.shira@gmail.com

\end{document}